\RequirePackage{fix-cm}
\documentclass[smallextended]{svjour3}       
\smartqed  
\usepackage{graphicx}
\usepackage{amscd}
\usepackage{amsmath}

\begin{document}
\title{Nonlinear elliptic equations with a singular perturbation on compact Lie groups and homogeneous spaces
}


\author{Weiping Yan \and Yong Li      
}


\institute{College of of Mathematics, Jilin University, Changchun
130012, P.R. China.\\
Beijing International Center for Mathematical Research, Peking University, Beijing 100871, P.R. China.\\
\email{yan8441@126.com}}

\date{Received: date / Accepted: date}

\maketitle

\begin{abstract}
This paper is devoted to the study of a class of singular perturbation elliptic type problems on compact Lie groups or homogeneous spaces $\mathcal{M}$.
By constructing a suitable Nash-Moser-type iteration scheme on compact Lie groups and homogeneous spaces, we overcome the clusters of ``small divisor'' problem, then the existence of solutions for nonlinear elliptic equations with a singular perturbation is established. Especially, if $\mathcal{M}$ is the standard torus $\textbf{T}^n$ or the spheres $\textbf{S}^n$, our result shows that there is a local uniqueness of spatially periodic solutions for nonlinear elliptic equations with a singular perturbation.

\end{abstract}

\section{Introduction and Main Results}
\label{intro}
The problem of solving nonlinear
elliptic equations with a singular perturbation inspired by the work of Rabinowitz\cite{Rabi2}.
He studied the solvability of the following equation with singular perturbation
\begin{eqnarray*}
-\sum_{i,j=1}^{n}(a_{i,j}(x)u_{x_j})_{x_i}+u=\varepsilon
f(x,u,Du,D^2u,D^3u),
\end{eqnarray*}
where $x=(x_1,x_2,\dots,x_n)\in\textbf{R}^n$, coefficients $a_{i,j}$
are periodic in $x_1,x_2,\dots,x_n$, $\varepsilon\in\textbf{R}$, the
function $f$ is also periodic in $x_1,x_2,\dots,x_n$. By employing the Nash-Moser iteration process, he
proved that above elliptic singular perturbation problem has a
uniqueness spatial periodic solution. Han, Hong and Lin\cite{Han}
partially extended the work of Rabinowitz\cite{Rabi2}, they
considered the following singular perturbation problem
\begin{eqnarray*}
-\triangle u+u+\varepsilon a(D^pu)=f(x),~~x\in\textbf{R}^2,
\end{eqnarray*}
where $p\geq 4$, the function $a(x)$ is smooth and $f(x)$ is
$(2\pi)^2$-periodic. Under some assumptions on $a(x)$ and $f(x)$, they
employed the Nash-Moser iteration process to prove that above
singular problem had spatial periodic solutions. For more related work, we refer to \cite{K,ph}.

In this paper, we consider the following nonlinear elliptic equations with singular perturbation
\begin{eqnarray}\label{E1-1}
-\triangle u+u+\varepsilon a(D^{2\varrho}u)=f(x,u),
\end{eqnarray}
where $\varrho\in\textbf{N}$, $x\in\mathcal{M}$, $\mathcal{M}$ is a compact Lie group or, more generally, a compact homogeneous space. The main difficulty
is the presence of arbitrarily ''small divisors'' in the series expansion of the solutions.
 The operator $\triangle$ is the Laplace-Beltrami operator defined with respect to a Riemannian metric compatible with the group structure. The nonlinearity is finitely differentiable and vanishes at
$\textbf{u}=0$ at least $2$. Classical examples of compact connected Lie groups are the standard torus $\textbf{T}^n$, the special orthogonal group $SO(n)$ and the special unitary group $SU(n)$. Examples of compact homogeneous space are the spheres $\textbf{S}^n$, the real and complex Grassmanians,
and the moving frames, namely, the manifold of the $k$-ples of orthonormal vectors in $\textbf{R}^n$ with the natural action of the orthogonal group $O(n)$. For more examples, see\cite{Br,He2}.

The information on the spectral analysis of the Laplace-Beltrami
operator can be provided by the presence of continuous symmetries expressed via a Lie
group action. When the action is transitive (up to isomorphism),
\begin{eqnarray*}
\mathcal{M}=(G\times\textbf{T}^n)/N,
\end{eqnarray*}
where $G$ is a simply connected compact Lie group, $\textbf{T}^n$ is a torus, and $N$ is a closed
subgroup of $G\times\textbf{T}^n$. The functions on $\mathcal{M}$ can be seen as functions defined
on $G\times\textbf{T}^n$ and invariant under the action of $N$, namely
\begin{eqnarray}\label{E1-7}
\textbf{L}^2(\mathcal{M})&=&\textbf{L}^2((G\times\textbf{T}^n)/N)\nonumber\\
&=&\{u\in\textbf{L}^2(G\times\textbf{T}^n)|u(xg)=u(x),~\forall x\in G\times\textbf{T}^n,~g\in N\}.~~~
\end{eqnarray}
Thus, the Laplace-Beltrami operator on $\mathcal{M}$ can be identified with the Laplace-
Beltrami operator on $G\times\textbf{T}^n$, acting on the functions invariant under $N$.

The eigenvalues and the eigenfunctions of the Laplacian on a
simply connected compact group $G$ are, respectively,
\begin{eqnarray*}
-|j_1+\rho|^2+|\rho|^2,~~and~~\textbf{e}_{j_1,\sigma}(x_1),~~x_1\in G,~~j_1\in\Lambda^+(G),~\sigma=1,\cdots,d_{j_1},
\end{eqnarray*}
where $\Lambda^+(G)$ is the cone generated by the natural combinations of the fundamental weights $w_i\in\textbf{R}^r$, $i=1,\cdots,r$. $r$ denotes the rank of the group, and $\rho:=\sum_{i=1}^rw_i$. The degeneracy of the eigenvalues is $d_{j_1}\leq|j_1+\rho|^{\dim(G)-r}$. Furthermore, there exits a constant $D:=D(G)\in\textbf{N}$ such that $-|j_1+\rho|^2+|\rho|^2\in\textbf{Z}D^{-1}$, $\forall j_1\in\Lambda^+(G)$. Using the Fubini theorem, $\textbf{L}^2(G\times\textbf{T}^n)=\textbf{L}^2(G)\times\textbf{L}^2(\textbf{T}^n)$. By (\ref{E1-7}), we conclude that the eigenvalues and the eigenfunctions of $-\triangle+1$ on $\mathcal{M}$ are, respectively
\begin{eqnarray*}
\omega_j^2:=|j_1+\rho|^2-|\rho|^2+|j_2|^2+1,~~\textbf{e}_{j,\sigma}(x)=\textbf{e}_{j_1,\sigma}(x_1)e^{ij_2\cdot x_2},~~x=(x_1,.x_2)\in G,
\end{eqnarray*}
where the index $j=(j_1,j_2)$ is restricted to a subset $\Lambda_{\mathcal{M}}\subset\Lambda^+(G)\times\textbf{Z}^n$, $j_1\in\Lambda^+(G)\times\textbf{T}^n$, $\sigma\subset[1,d_{j}],~d_j:=\dim(\mathcal{M}_j),~\mathcal{M}_j\subset\mathcal{M},~d_j\leq d_{j_1}$. This property is crucial to Lemma 12 in section 3.

Rescaling in (\ref{E1-1}) amplitude $u(x)\mapsto\delta u(x)$, $\delta>0$, we solve the following problem
\begin{eqnarray}\label{E1-2}
-\triangle u+u+\varepsilon a(D^{2\varrho}u)=\varepsilon f(\delta,u),
\end{eqnarray}
where $a(s):=as^{p}$, $f(\delta,u):=b(x)s^p+O(\delta)$, $1\leq p\leq k$ and $\varepsilon=\delta^{p-1}$.

In our paper, we will divide into two cases to discuss the existence of solutions for (\ref{E1-2}). The first case is
$a(x)=ax$, where $a\neq0$ is a constant, then the ``small divisor'' phenomenon appears. The second case is
$a(\cdot)\in\textbf{C}^{k}(\textbf{R})$. The second case is simpler than the first case, and we can use
the Nash-Moser iteration scheme constructed in the first case to solve it. In what follows, we deal with the first case, i.e. $a((-1)^{\varrho}\triangle^{\varrho}u)=(-1)^{\varrho}a\triangle^{\varrho}u$. Thus we can rewrite (\ref{E1-2}) as
\begin{eqnarray}\label{E1-2R}
-\triangle u+u+(-1)^{\varrho}\varepsilon a\triangle^{\varrho}u=\varepsilon f(\delta,u).
\end{eqnarray}
Assume that $a$ is an irrational number and diophantine, i.e. there are constants $\gamma_0>0$, $\tau_0>1$, such that
\begin{eqnarray}\label{E1-4}
|m-a n|\geq\frac{\gamma_0}{|n|^{\tau_0}},~~\forall(m,n)\in\textbf{Z}^2\backslash \{(0,0)\}.
\end{eqnarray}
Then using Lemma 2 (see section 2) there exist $\gamma>0$ and $\tau>0$ such that the first order Melnikov nonreonance condition
\begin{eqnarray}\label{E1-5}
|\omega^2_j+1-\varepsilon a\omega_j^{2\varrho}|\geq\frac{\gamma}{|j+\overrightarrow{\rho}|^{\tau}},
\end{eqnarray}
where $\omega_j^2;=|j_1+\rho|^2-|\rho|^2+|j_2|^2$, $j=(j_1,j_2)\in\Lambda_{\mathcal{M}}\subset\Lambda^+(G)\times\textbf{Z}^n$ and $\overrightarrow{\rho}=(\rho,0)$.

In this paper, we make more general assumptions on nonlinear terms $f$, which include the standard tame estimates and Taylor tame estimates.
We assume that the nonlinear terms $f\in\textbf{C}^k(\mathcal{M}\times\textbf{R},\textbf{R})$, $f(0,0)=0$, $\partial_u f(x,0)=\cdots=(\partial_u^{p-1})f(x,0)=0$, $\partial_u^{p}f(x,0)\neq0$, $1\leq p\leq k$, $k\geq2$ and
\begin{eqnarray}
\label{E1-6}
\|\partial_uf(x,u')u\|_s\leq c(s)(\|u\|_s^{p-1}+\|u'\|_s\|u\|^{p-1}_{s_0}),
\end{eqnarray}
\begin{eqnarray}
\label{E1-8}
\|f(x,u+u')&-&f(x,u')-D_uf(x,u')u\|_s\nonumber\\
&\leq& c(s)(\|u'\|_{s}\|u\|_{s_0}^{p-1}+\|u\|_{s_0}\|u\|_{s}^{p-1}),~~~~
\end{eqnarray}
where $s>s_0>0$, $p>1$, $\forall u,u'\in\textbf{H}_s$ such that $\|u\|_{s_0}\leq1$ and $\|u'\|_{s_0}\leq1$.
In particular, for $s_0=s$,
\begin{eqnarray*}
\|f(x,u+u')-f(x,u')-D_uf(x,u')u\|_s\leq c(s)\|u\|_{s}^{p}.
\end{eqnarray*}
In fact, when $p=2$, assumption (\ref{E1-6}) and (\ref{E1-8}) are natural for $f\in\textbf{C}^k(\mathcal{M}\times\textbf{R},\textbf{R})$, which are tame estimates and Taylor tame estimates, respectively.

Our main results are based on the Nash-Moser iterative scheme, which is firstly introduced by Nash\cite{Nash} and Moser\cite{Moser}; see\cite{H} for more details.
Recently, Berti and Procesi\cite{Berti1} developed suitable linear and nonlinear harmonic analysis on compact Lie groups and homogeneous spaces, and via the technique and the Nash-Moser implicit function theorem, they found a family of time-periodic solutions of nonlinear Schr\"{o}dinger equations and wave equations. Inspired by the work of \cite{Berti1,Bourgain1,Rabi2,Yan}, we will construct a new suitable Nash-Moser iteration scheme to study the elliptic-type singular perturbation problems (\ref{E1-1}) on compact Lie groups and homogeneous spaces. Meanwhile, Theorems 3-4 allow more general $\textbf{C}^k$ nonlinearities on a higher dimensional space than the work of \cite{Han,Rabi2}. Since the proof process of Theorem 3-4 is similar with Theorem 1-2, we omit them.
For a general case, we assume that the nonlinear terms satisfy (\ref{E1-6})-(\ref{E1-8}).

To carry out the Nash-Moser iteration scheme, we also need to introduce the Banach scale of the Sobolev spaces on a group. Let $\mathcal{M}=(G\times\textbf{T}^n)/N$ be a homogeneous space, where $G$ is a simply connected Lie group of dimension $d$ and rank $r$. By Theorem 5 in section 2 (Peter-Weyl theorem), we have the orthogonal decomposition
\begin{eqnarray*}
\textbf{L}^2:=\textbf{L}^2(\mathcal{M},\textbf{C})=\bigoplus_{j\in\Lambda_{\mathcal{M}}}\mathcal{N}_j.
\end{eqnarray*}
The Fourier series of $u\in\textbf{L}^2$ is defined by
\begin{eqnarray*}
u=\sum_{j\in\Lambda_{\mathcal{M}}}u_j,
\end{eqnarray*}
where $u_j:=\Pi_{\mathcal{N}_j}u$ and $\Pi_{\mathcal{N}_j}:\textbf{L}^2\longrightarrow\mathcal{N}_j$ are the spectral projectors, $\Lambda_{\mathcal{M}}:=\{j\in\Lambda^+\times\textbf{Z}^n~~such~~that~~\mathcal{N}_j\neq\{0\}\}$ is closed under sum.

More precisely, for $\forall 1\leq d_j'\leq d_j$, we have
\begin{eqnarray*}
u(x)=\sum_{j\in\Lambda_{\mathcal{M}}}\sum_{\sigma=1}^{d'_j}u_{j,\sigma}\textbf{e}_{j,\sigma}
\end{eqnarray*}
with the norm
\begin{eqnarray*}
\|u(x)\|^2_{\textbf{L}^2}=2\pi\sum_{j\in\Lambda_{\mathcal{M}}}\sum_{\sigma=1}^{d'_j}|u_{j,\sigma}|^2.
\end{eqnarray*}
We define the Sobolev scale of Hilbert spaces
\begin{eqnarray*}
\textbf{H}_s:=\textbf{H}_s(\mathcal{M},\textbf{C})=\{u=\sum_{j\in\Lambda_{\mathcal{M}}}u_j|\|u\|_s^2:=\sum_{j\in\Lambda_{\mathcal{M}}}e^{2|j+\overrightarrow{\rho}|s}\|u\|_{\textbf{L}^2}^2<+\infty\},
\end{eqnarray*}
where $\overrightarrow{\rho}=(\rho,0)\in\Lambda^+\times\textbf{Z}^n$. It is obvious that $\textbf{H}^0=\textbf{L}^2$.
Since $\mathcal{M}$ is a compact $C^{\infty}$-Riemannian manifold without boundary, for any $s\in\textbf{N}$, $\textbf{H}_s$ is equivalent to the usual Sobolev sapce
\begin{eqnarray*}
H_s=\{u\in\textbf{L}^2|D^{\alpha}u\in\textbf{L}^2,\forall|\alpha|\leq s, \|u\|_s^2:=\sum_{|\alpha|\leq s}\|D^{\alpha}u\|_{\textbf{L}^2}^2\}.
\end{eqnarray*}

For the case $a(x)=ax$ in (\ref{E1-1}), we have the following result.
\begin{theorem}
Let $\tau,\kappa_0,\delta>0$ and $0<\sigma_{0}(\mathcal{M})<\bar{\sigma}(\mathcal{M})<\sigma(\mathcal{M})<k(\mathcal{M})-1$. Assume that $a>0$ is diophantine. For $\delta_0>0$, $s_0:=\bar{\sigma}(\mathcal{M})$, $k:=k(\mathcal{M})\in\textbf{N}$ and $f\in\textbf{C}^k$ satisfying (\ref{E1-6})-(\ref{E1-8}), Then there exists a positive measure Cantor set $\mathcal{C}\subset[0,\delta_0]$ such that, $\forall a\in\mathcal{C}$, $u_{\delta}(x,\varepsilon)$ is a local uniqueness solution of (\ref{E1-2R}).
Furthermore, there exists a curve
\begin{eqnarray*}
u\in\textbf{C}^1([0,\delta_0];\textbf{H}_{s_0})~~with~\|u(\delta)\|_{s_0}=O(\delta).
\end{eqnarray*}
 \end{theorem}

For the second case, we consider equation (\ref{E1-2}) and assume that $a\in\textbf{C}^k(\textbf{R})$,
$a(0)=0$, and
\begin{eqnarray}
\label{E1-6R}
&&\|\partial_ua(u')u\|_s\leq c(s)(\|u\|_s^{p-1}+\|u'\|_s\|u\|^{p-1}_{s_0}),\\
\label{E1-8R}
&&\|a(u+u')-a(u')-D_ua(u')u\|_s\leq c(s)(\|u'\|_{s}\|u\|_{s_0}^{p-1}+\|u\|_{s_0}\|u\|_{s}^{p-1}),~~~~~~~
\end{eqnarray}
where $s>s_0>0$, $1<p\leq k$, $\forall u,u'\in\textbf{H}_s$ such that $\|u\|_{s_0}\leq1$ and $\|u'\|_{s_0}\leq1$.
In particular, for $s_0=s$,
\begin{eqnarray*}
\|a(u+u')-a(u)-D_ua(u)u\|_s\leq c(s)\|u\|_{s}^{p}.
\end{eqnarray*}

Then we have

\begin{theorem}
Let $\tau,\kappa_0>0$ and $0<\sigma_{0}(\mathcal{M})<\bar{\sigma}(\mathcal{M})<\sigma(\mathcal{M})<k(\mathcal{M})-1$. There exist $s_0:=\bar{\sigma}(\mathcal{M})$ and $k:=k(\mathcal{M})\in\textbf{N}$ such that for $f,a\in\textbf{C}^k$ satisfying (\ref{E1-6})-(\ref{E1-8R}), equation (\ref{E1-2}) has a solution $u(x)\in\textbf{H}_{s_0}$.
\end{theorem}
The proof of Theorem 2 is similar to the proof of Theorem 1, hence we omit it.

Especially, if $\mathcal{M}$ is the standard torus $\textbf{T}^n$ or the spheres $\textbf{S}^n$, we obtain the existence of spatially periodic solutions for elliptic equation (\ref{E1-1}). We also need to divide into two cases to discuss. For the first case, we have
\begin{theorem}
Let $\tau,\kappa_0>0$ and $0<\sigma_{0}<\bar{\sigma}<\sigma<k-1$.
Assume that $a>0$ is diophantine. For $\delta_0>0$, $s_0:=\bar{\sigma}(\mathcal{M})$, $k:=k(\mathcal{M})\in\textbf{N}$ and $f\in\textbf{C}^k$ satisfying (\ref{E1-6})-(\ref{E1-8}), Then there exists a positive measure Cantor set $\mathcal{C}\subset[0,\delta_0]$ such that, $\forall a\in\mathcal{C}$,
$u=u_{\delta}(x,\varepsilon)$ is a unique spatially periodic solution of (\ref{E1-1}).
Furthermore, there exists a curve
\begin{eqnarray*}
u\in\textbf{C}^1([0,\delta_0];\textbf{H}_{s_0})~~with~\|u(\delta)\|_{s_0}=O(\delta).
\end{eqnarray*}
\end{theorem}
For the second case, we have
\begin{theorem}
Let $\tau,\kappa_0>0$ and $0<\sigma_{0}<\bar{\sigma}<\sigma<k-1$.  There exist $s_0:=\bar{\sigma}$ and $k\in\textbf{N}$ such that $\forall f,a\in\textbf{C}^k$ satisfying (\ref{E1-6})-(\ref{E1-8R}). Then equation (\ref{E1-1}) has a spatially periodic solution $u(x)\in\textbf{H}_{s_0}$.
\end{theorem}

The structure of the paper is as follows: In next section, we present some notations related to Lie group, homogeneous spaces and corresponding Laplace-Beltrami operator properties. Section 3 is devoted to the proof of Theorem 1, where we construct a suitable Nash-Moser iteration scheme. In the last section, we will prove a main Lemma (Lemma 12), which deals with the estimate of the linearized operators and plays a crucial role in the Nash-Moser iteration. The measure estimates is given in the appendix.

\section{Preliminaries}
In this section, we recall some basic conceptions and results in the representation theory of Lie group and homogeneous space, which can be found in the books  \cite{Br,Fa,Pro} and the paper \cite{Berti1}. Let $G$ be a compact topological group, and let $\textbf{L}^2(G):=\textbf{L}^2(G,\textbf{C})$ be the Lebesgue
space defined with respect to the normalized Haar measure $\mu$ of $G$. $(V,\rho_V)$ denotes a finite-dimensional unitary representation of $G$. It is a continuous homomorphism $x\mapsto\rho_V(x)$ which maps $G$ into the group of unitary transformations $U(V)\subset End(V)$; here $V$ denotes a finite-dimensional complex vector space. For fixed $\{v_1,\cdots,v_n\}$ (orthonormal basis) of $V$, we can describe the presentation by the unitary matrices
\begin{eqnarray}\label{E2-1}
U(x):=U^{V}(x):=\{U^V_{l,k}(x)\}=\{(\rho_V(x)v_l,v_k)\},~~l,~k=1,\cdots,n:=\dim(V).~~~
\end{eqnarray}

The following Peter-Weyl Theorem gives the Fourier analysis on the group. In the case of the standard torus, the irreducible representations of a group play the role of the exponential basis.
\begin{theorem}
Let $\hat{G}$ be the set of equivalence classes of irreducible unitary representations of the compact group $G$, for each $j\in\hat{G}$, let $\mathcal{M}_j:=\mathcal{M}_{V_j}$. Then the Hilbert decomposition holds
\begin{eqnarray*}
\textbf{L}^2(G)=\widehat{\bigoplus}_{j\in\hat{G}}\mathcal{M}_j.
\end{eqnarray*}
For $f\in\textbf{L}^2(G)$, we have the $\textbf{L}^2$ convergent ``Fourier series''
\begin{eqnarray*}
f(x)=\sum_{j\in\hat{G}}tr(f_j\textbf{e}_j(x)),~~f_j:=\int_{G}df(x)\bar{\textbf{e}}_j(x)\mu,
\end{eqnarray*}
where $\textbf{e}_j(x):=(\dim V_j)^{\frac{1}{2}}U_j(x)$ and the matrices $U_j(x):=U^{V_j}(x)$ are defined in (\ref{E2-1}). Here the matrix $\bar{\textbf{e}}_j(x)$ is the complex conjugate of $\textbf{e}_j(x)$, and $f_j$ are the Fourier coefficient of $f(x)$.
\end{theorem}

By the Schur orthogonality relations
\begin{eqnarray*}
\int_{G}tr(A\rho_V(x))\overline{tr(B\rho_V(x))}d\mu(x)=\frac{tr(AB^{\dag})}{\dim(V)},~~\forall~A,B\in End(V),
\end{eqnarray*}
we have that $\textbf{e}_{j,\sigma}(x)$ (the matrix coefficients of $\textbf{e}_{j}(x)$), $\sigma=1,\cdots,\dim V_j^2$ form an $L^2$-orthonormal basis for $\mathcal{M}_j$.

Next we introduce some properties of Laplace-Beltrami operator on the compact Lie groups $\mathcal{G}=(G\times\textbf{T}^n)/N$, where $G$ is simply connected and $N$ is finite and central. Let $G$ be a simply connected compact Lie group of simple type. Define a Riemannian metric on $G$ by
\begin{eqnarray*}
-(X,Y):=tr(Ad(X)\circ Ad(Y)),
\end{eqnarray*}
which is negative definite of the killing form, where $Ad(X)(\cdot):=[X,\cdot]$. Thus we can define the Laplace-Beltrami operator $\triangle$ on $G$ with respect to this metric. The following two results are taken from the book \cite{Pro}.
\begin{theorem}
For a simply connected compact Lie group $G$ of rank $r$, there is a one-to-one correspondence between the set of equivalence classes $\hat{G}$ of irreducible unitary representations and a discrete cone
\begin{eqnarray*}
\Lambda^+:=\Lambda^+(G)=\{j=\sum_{i=1}^rn_iw_i,n_i\in\textbf{N}\}\subset\textbf{R}^r
\end{eqnarray*}
generated by $r$ independent vectors $w_i\in\textbf{R}^r$.
\end{theorem}
Above result describes all the irreducible representations. Here the rank of a Lie group $G$ is defined by the dimension of any maximal connected commutative subgroup of $G$ (maximal torus). The $\{w_1,\cdots,w_r\}$ are called the fundamental weights of the group, and $\Lambda^+$ is called the cone of dominant weight. The irreducible representation of $G$ corresponding to the dominant weight $j=0$ is the trivial representations on $V_0=\textbf{C}$.

The matrix coefficients of an irreducible representation are eigenfunctions of the Laplacian. Due to the Laplacian is a real operator, $\bar{\textbf{e}}_{j,\sigma}(x)$ is an eigenvector of $\triangle$ if $\textbf{e}_{j,\sigma}(x)\in\mathcal{M}_j$ is an eigenvector with the same eigenvalue. Thus $\bar{\textbf{e}}_{j,\sigma}(x)\in\mathcal{M}_{j'}$ for some $j'\in\Lambda^+$. Moreover, since the matrix $\bar{\textbf{e}}_{j,\sigma}(x)$ is the dual representation on $V_j^*$ of the matrix $\textbf{e}_j(x)$, we have $V_{j'}=V_j^*$.
\begin{theorem}
Each $\mathcal{M}_j$ is an eigenspace of the Laplace Beltrami operator $\triangle$ with eigenvalue
\begin{eqnarray*}
-|j+\rho|^2+|\rho|^2,~~~~~\rho:=\sum_{i=1}^rw_i,
\end{eqnarray*}
where $d_j:=\dim(\mathcal{M}_j)\leq|j+\rho|^{\dim(G)-r}$.
\end{theorem}
Set the positive simple roots $\alpha_1,\cdots,\alpha_r\in\textbf{R}^r$ of the group $G$, then the eigenvalues and the eigenfunctions of the Laplace operator can be described, see\cite{Pro} for more details. They satisfy the relations
\begin{eqnarray*}
(w_i,\alpha_j)=\frac{1}{2}\delta_{i,j}|\alpha_j|^2,~~\forall~i,j=1,\cdots,r,
\end{eqnarray*}
where $\delta_{i,j}$ denotes the Kronecker symbol. Define the cone
\begin{eqnarray*}
\mathcal{R}^+:=\{\alpha=\sum_{i=1}^{r}n_i\alpha_i,n_i\in\textbf{N}\}\subset\textbf{R}^r
\end{eqnarray*}
generated by the natural combinations of the positive simple roots. We define the lattice
\begin{eqnarray*}
\Lambda:=\{j=\sum_{i=1}^{r}n_iw_i,n_i\in\textbf{Z}\}\subset\textbf{R}^r
\end{eqnarray*}
generated by the fundamental weights $w_1,\cdots,w_r$ and $\Lambda^{++}:=\rho+\Lambda^+$.

\textbf{Definition 1.}
For $i,j\in\Lambda$, we say that $i\geq j$ if $i=j+\alpha$ for some $\alpha\in\mathcal{R}^+$.

The following two results are taking from\cite{Berti1}, and we omit the proof.
\begin{lemma}
There is $c\in(0,1)$ such that $\forall i\geq j$, $i\in\Lambda^+$, $j\in\Lambda$ with $(\rho,j)>0$, one has $|i+\rho|\geq c|j+\rho|$.
\end{lemma}
\begin{lemma}
For any simply connected Lie group $\mathcal{G}$, there is $D\in\textbf{N}$ such that $(w_i,w_j)\in D^{-1}\textbf{Z}$. Hence
\begin{eqnarray*}
|j|^2,~|j+\rho|^2,~(\rho,j)\in D^{-1}\textbf{Z},~~\forall j\in\Lambda^+.
\end{eqnarray*}
\end{lemma}
The eigenspaces of the Laplace operator on $G\times\textbf{T}^n$ are
\begin{eqnarray*}
\mathcal{M}_{j_1}e^{ij_2\cdot x_2}~~with~(x_1,x_2)\in G\times\textbf{T}^n,~~(j_1,j_2)\in\Lambda^+\times\textbf{Z}^n,
\end{eqnarray*}
the eigenfunctions $\textbf{e}_{j_1,\sigma}(x_1)e^{ij_2\cdot x_2}$, $1\leq\sigma\leq d_j$, and the eigenvalues $-|j_1+\rho|^2+|\rho|^2-|j_2|^2$.
\begin{theorem}
Let $H$ be a closed subgroup of a Lie group $\mathcal{G}$. Then there is a unique manifold structure on the quotient space $\mathcal{G}/H$, such that the projection map
$\pi:\mathcal{G}\rightarrow\mathcal{G}/H$ is a smooth submersion. Moreover, given a biinvariant metric on $\mathcal{G}$, the projection $\pi$ induces on $\mathcal{G}/H$ a Riemannian structure such that the Laplace-Beltrami operator on $C^{\infty}(\mathcal{G}/H,\textbf{C})$ is identified with the Laplace-Beltrami operator on
\begin{eqnarray*}
C_{inv}^{\infty}(\mathcal{G},\textbf{C}):=\{f\in C^{\infty}(\mathcal{G},\textbf{C})~such~that~f(x)=f(xg),~\forall x\in\mathcal{G},~~g\in H\},
\end{eqnarray*}
and the diagram commutes:
\[ \begin{CD}
C^{\infty}(\mathcal{G}/H,\textbf{C}) @>\text{$\pi^*$}>> C_{inv}^{\infty}(\mathcal{G},\textbf{C})\\
@V\triangle_{\mathcal{G}/H} VV @V \triangle_{G} VV\\
C^{\infty}(\mathcal{G}/H,\textbf{C})@>\text{$\pi^*$}>>C_{inv}^{\infty}(\mathcal{G},\textbf{C}).
\end{CD} \]
\end{theorem}
The action $(g,x)\mapsto gx$ of a group $\mathcal{G}$ on a set $X$ is called transitive if, $\forall x\in X$, the orbit $\mathcal{O}(x):=\{gx\in X,g\in\mathcal{G}\}=X$.

\textbf{Definition 2.}
A compact manifold $\mathcal{M}$ is said to be homogeneous if there is a compact Lie group
$G$ which acts on $\mathcal{M}$ transitively and differentiably; that is, for each $g\in G$, the map
$x\mapsto gx$ is differentiable in $\mathcal{M}$.

The action of $G\times\textbf{T}^n$ on any $p\in\mathcal{M}$ induces a diffeomorphism $\mathcal{M}\leftrightarrow(G\times\textbf{T}^n)/N$, where $N:=N_1\mathcal{G}_p$, $N_1$ is th finite central subgroup, $\mathcal{G}_p$ is the stabilizer of $p$, the group $\mathcal{G}=(G\times\textbf{T}^n)/N_1$. Theorem 8 shows that a biinvariant metric on $G\times\textbf{T}^n$ induces a metric on $(G\times\textbf{T}^n)/N$ and, then, on $\mathcal{M}$ (see \cite{Besse}).

By Theorem 5 (Peter-Weyl theorem), we have the spectral theory of the Laplace-Beltrami operator on a compact homogeneous space.
\begin{theorem}
The following sum decomposition holds:
\begin{eqnarray*}
\textbf{L}^2(\mathcal{M})=\widehat{\bigoplus}_{j\in\Lambda_{\mathcal{M}}}\mathcal{N}_j.
\end{eqnarray*}
A basis for $\mathcal{N}_j\subset\mathcal{M}_j$ is, up to a reordering of the index $\sigma$,
\begin{eqnarray*}
\textbf{e}_{j,\sigma}(x)=\textbf{e}_{j_1,\sigma}(x_1)e^{ij_2\cdot x_2},~~\sigma=1,\cdots,d'_j,
\end{eqnarray*}
for some $1\leq d'_j\leq d_j$, where the subspace of functions $\mathcal{N}_j\subset\mathcal{M}_j:=\mathcal{M}_{j_1}e^{ij_2\cdot x_2}$ defined by
\begin{eqnarray*}
\mathcal{N}_j:=Span\{(\rho_{v_j}(x)w_k,v_l),k=1,\cdots,\dim(W),l=1,\cdots,\dim(V)\},
\end{eqnarray*}
the subspace $W_j:=\{w\in V_j|\rho_{V_j}(g)w=w,\forall g\in H\}\subset V_j$, $(v_l)_{l=1,\cdots,\dim(V_j)}$ and $(w_l)_{l=1,\cdots,\dim(W_j)}$ are a basis of $V_j$ and $W_j$, respectively. Moreover, each $\mathcal{N}_j$ is an eigenspace of the Laplacian with dimension $\dim\mathcal{N}_j\leq\dim\mathcal{M}_j$ and $\dim\mathcal{N}_j=\dim V_j\dim W_j$.
\end{theorem}
Let $a=tr(A\rho_{V_i})\in\mathcal{M}_i$ and $b=tr(B\rho_{V_j})\in\mathcal{M}_j$
denote two eigenfunctions of the Laplace-Beltrami operator $\triangle$.
The product is given by
\begin{eqnarray*}
ab=tr(A\otimes B\rho_{V_i\otimes V_j}).
\end{eqnarray*}
Then $V_i\otimes V_j$ can be expressed as the direct sum of irreducible representations
\begin{eqnarray}\label{E2-2}
V_i\otimes V_j=\bigoplus_{l\in\Lambda^+}V_l^{c_{i,j}^l},~~c^l_{i,j}\in\textbf{Z}\cup\{0\},
\end{eqnarray}
where $c_{i,j}^l$ are called the Clebsch-Gordan coefficients of the group.

Using a theorem of Cartan (see \cite{Pro}, p.345), one can verify that the product of two eigenfunctions is a finite sum of eigenfunctions.
\begin{lemma}
Let $a\in\mathcal{M}_i$ and $b\in\mathcal{M}_j$. Then $ab\in\bigoplus_{l\leq i+j}\mathcal{M}_l$.
\end{lemma}
The $\textbf{L}^2$-orthogonal projection of $ab$ on the eigenspace $\mathcal{M}_l$ is
\begin{eqnarray*}
\prod_{\mathcal{M}_l}ab=\sum_{s\leq c_{ij}^l}tr((A\otimes B)|_{l,s}\rho_{V_l}),
\end{eqnarray*}
where $(A\otimes B)|_{l,s}$ denotes the restriction of $A\otimes B$ to the $s$th copy of $V_l$ in (\ref{E2-2}).

Specially, if $A=B=Id$, we obtain the formula for the characters $\chi_i:=\chi_{V_i}$, namely,
\begin{eqnarray*}
\chi_i\chi_j=\sum_{l\leq i+j}c_{ij}^l\chi_l.
\end{eqnarray*}

\begin{lemma}
For $s\geq s_0>\frac{\dim(\mathcal{M})}{2}$, $\forall u_1,u_2\in\textbf{H}_s$, there hold:
\begin{eqnarray*}
&&\|u_i\|_{\textbf{L}^{\infty}}\leq c(s)\|u_i\|_s,~~i=1,2,\\
&&\|u_1u_2\|_s\leq c(s)\|u_1\|_s\|u_2\|_s,\\
&&\|u_1u_2\|_s\leq c(s,s_0)(\|u_1\|_s\|u_2\|_{s_0}+\|u_1\|_{s_0}\|u_2\|_{s_1}).
\end{eqnarray*}
\end{lemma}

\begin{lemma}
We define $J:=\Lambda\times\textbf{Z}^n$ and $J^+:=\Lambda^+\times\textbf{Z}^n$, for given $a\in\mathcal{N}_j$ and $b\in\mathcal{N}_{j'}$. Then
\begin{eqnarray*}
ab\in\bigoplus_{\bar{j}\in D(j,j')}\mathcal{N}_{\bar{j}},
\end{eqnarray*}
where for $j=(j_1,j_2)$, $j'=(j'_1,j'_2)$, $\tilde{j}=(\tilde{j}_1,\tilde{j}_2)$,
\begin{eqnarray*}
D(j,j'):=\{\tilde{j}\in J^+|\tilde{j}_1\leq j_1+j'_1,\tilde{j}_2=j_2+j'_2\},
\end{eqnarray*}
\end{lemma}
Here for given $j=(j_1,j_2),~j'=(j_1',j_2')\in J$, we say that $j\geq j'$ if $j_1\geq j_1'$ and $|j_2|\geq|j_2'|$.

The following result shows that the embedding $\textbf{H}_s\hookrightarrow\textbf{C}(\mathcal{M})$ for $2s>\dim(G\times\textbf{T}^n)>\dim(\mathcal{M})$. We denote $J^+:=\Lambda_{\mathcal{M}}$. By a small modification of the proof of Lemma 2.15 in \cite{Berti1}, we have
\begin{lemma}
Let $2s>d+n+1$. For $u\in\bigoplus_{j\geq j_0,j\in J^+}\mathcal{N}_j$ with $j_0=(j_{01},j_{02})\in\Lambda^+(G)\times\textbf{Z}^n$, and $(\rho,j_{01})\geq0$, we have
\begin{eqnarray*}
\|u\|_{\textbf{L}^{\infty}}\leq c(s)\|u\|_se^{-(s-\frac{d+n+1}{2})|j_0|}.
\end{eqnarray*}
\end{lemma}
Due to the orthogonal splitting
\begin{eqnarray*}
\textbf{H}_s=\bigoplus_{j\in J^+}\mathcal{N}_j,
\end{eqnarray*}
we identify a linear operator $A$ acting on $\textbf{H}_s$ with its matrix representation $A=(A_j^{j'})_{j,j'\in J^+}$ with blocks $A_j^{j'}\in\mathcal{L}(\mathcal{N}_{j'},\mathcal{N}_{j})$.

We define the polynomially localized block matrices
\begin{eqnarray*}
\mathcal{A}_s:=\{A=(A_j^{j'})_{j,j'\in J^+}:|A|^2_s:=\sup_{j\in J^+}\sum_{j'\in J^+}e^{2s|j-j'|}\|A_j^{j'}\|_0^2<\infty\},
\end{eqnarray*}
where $\|A_j^{j'}\|_0:=\sup_{u\in\mathcal{N}_{j'},\|u\|_0=1}\|A_j^{j'}u\|_0$ is the $\textbf{L}^2$-operator norm in $\mathcal{L}(\mathcal{N}_{j'},\mathcal{N}_{j})$. If $s'>s$, then these holds $\mathcal{A}_{s'}\subset\mathcal{A}_s$.

The next lemma (see \cite{Berti1}) shows the algebra property of $\mathcal{A}_s$ and interpolation inequality.
\begin{lemma}
There holds
\begin{eqnarray}\label{E2-4}
&&|AB|_s\leq c(s)|A|_s|B|_s,~~\forall A,B\in\mathcal{A}_s,~~~s>s_0>\frac{r+n+1}{2},\\
\label{E2-5}
&&|AB|_s\leq c(s)(|A|_s|B|_{s_0}+|A|_{s_0}|B|_s),~~s\geq s_0,\\
\label{E2-6}
&&\|Au\|_s\leq c(s)(|A|_s\|u\|_{s_0}+|A|_{s_0}\|u\|_s),~~\forall u\in\textbf{H}_s,~~s\geq s_0.
\end{eqnarray}
\end{lemma}
By Lemma 7, we can get, $\forall m\in\textbf{N}$,
\begin{eqnarray}\label{E2-7}
&&|A^m|_s\leq c(s)^{m-1}|A|_s^{m},\\
\label{E2-8}
&&|A^m|_s\leq m(c(s)|A|_{s_0})^{m-1}|A|_s.
\end{eqnarray}
The next two lemmas can be obtained by a small modification of the proof of Lemma 2.18 and Proposition 2.19 in \cite{Berti1}, so we omit it.
\begin{lemma}
Let $A\in\mathcal{A}_s$, $\Omega_1,\Omega_2\subset J^+$, and $\Omega_1\cap\Omega_2=\emptyset$. Then
\begin{eqnarray*}
\|A_{\Omega_2}^{\Omega_1}\|_0\leq c(s)|A|_sd^{-1}(\Omega_1,\Omega_2)^{2s-(r+n+1)},
\end{eqnarray*}
where $r+n+1$ is the dimension of $J^+$.
\end{lemma}
Since $\textbf{H}_s$ is an algebra, for each $b\in\textbf{H}^s$ defines the multiplication operator
\begin{eqnarray}\label{E2-3}
u(x)\mapsto b(x)u(x),~~\forall u\in\textbf{H}_s,
\end{eqnarray}
which is represented by $(B_j^{j'})_{j,j'\in J^+}$ with $B_j^{j'}:=\Pi_{\mathcal{N}_j}b(x)|_{\mathcal{N}_{j'}}\in\mathcal{L}(\mathcal{N}_{j'},\mathcal{N}_j)$.

Using Lemmas 5-6, we obtain
\begin{lemma}
If $b\in\textbf{H}_s$ is real, then the matrix $(B_j^{j'})_{j,j'\in J^+}$ is self-adjoint, i.e. $(B^{j'}_j)^{\dag}=(B_{j'}^j)$, and $\forall2s\geq d+n+1$,
\begin{eqnarray*}
\|B_j^{j'}\|_0\leq c(s)\|b\|_se^{-(s-\frac{d+n+1}{2})|j-j'|}.
\end{eqnarray*}
\end{lemma}
We need to consider restricted matrices. Given a set of indexes $l\subset J^+$, we define
\begin{eqnarray*}
\mathcal{A}_s(l):=\{A=(A_j^{j'})_{j,j'\in J^+}:(A_j^{j'})^{\dag}=A_j^{j'},
|A|^2_s:=\sup_{j\in l}\sum_{j'\in l}e^{2s|j-j'|}\|A_j^{j'}\|_0^2<\infty\}.
\end{eqnarray*}
The next two lemmas can be seen as the corollaries of Lemma 2.9 (see \cite{Berti1}).
\begin{lemma}
For real functions $b\in\textbf{H}_{s+s'}$ with $2s'\geq d+r+2n+3$ , the matrix $(B_j^{j'})_{j,j'\in J^+}$ representing the multiplication operator (\ref{E2-3}) is self-adjoint, it belongs to the algebra of polynomially localized matrices $\mathcal{A}_s$, and we have
\begin{eqnarray*}
|B|_s\leq K(s)\|b\|_{s+s'}.
\end{eqnarray*}
\end{lemma}
\begin{lemma}
For $A=(A_j^{j'})_{j,j'\in J^+}\in\mathcal{A}_s$, its restriction $A_{l}=(A_j^{j'})_{j,j'\in l}\in\mathcal{A}_s(l)$ satisfies $|A_l|_s\leq|A|_s$. On the other hand, any $A\in\mathcal{A}_s(l)$ can be extended to a matrix in $\mathcal{A}_s$ by setting $A_j^{j'}=0$ for $j,j'\in l$ without changing the norm $|A|_s$.
\end{lemma}
Lemma 11 tells us that all the properties (algebra,interpolation, etc) hold for $\mathcal{A}_s(l)$ with constants in dependent of $l$. We use $I_l$ to denote the projectors
\begin{eqnarray*}
\Pi_l:\textbf{H}_s\longrightarrow\textbf{H}_l:=\bigoplus_{j\in l\cap J^+}\mathcal{N}_j~~satisfy~|I_l|=1,~\forall s\geq0.
\end{eqnarray*}

\section{Nash-Moser-type iteration scheme}
Let $(X_s,\|\cdot\|_s)_{s\geq0}$ be a scale of Banach spaces such that
\begin{eqnarray*}
\forall s\leq s',~~X_{s'}\subseteq X_s,~~\|u\|_s\leq\|u\|_{s'},~~\forall u\in X_{s'}.
\end{eqnarray*}
We define the finite dimensional subspaces
\begin{eqnarray*}
\textbf{H}_s^{(N_i)}:=\bigoplus_{j\in J_{N_i}^+}\mathcal{N}_j\subset\cap_{s\geq0}X_s,
\end{eqnarray*}
where $J_N^+:=\{j\in J^+||j+\overrightarrow{\rho}|\leq N_i\}$, $X_s=\textbf{H}_s(\mathcal{M},\textbf{R})$, $\forall s\leq k$, $i$ denotes the ``$i$''th iterative step. For a given suitable $N_0>1$, we take $N_i\leq N_{i+1}$ and and $N_i=N_0^i$, $\forall i\in\textbf{N}$.

Let $(\textbf{H}_s^{(N_i)})_{N_i\geq0}$ be an increasing family of closed subspaces of $\cap_{s\geq0}X_s$ with projectors
$\Pi^{(N_i)}:X_s\longrightarrow\textbf{H}_s^{(N_i)}$ satisfying the ``smoothing'' properties:
\begin{eqnarray}
&&\|\Pi^{(N_i)}u\|_{s+d}\leq N_i^d\|u\|_s,~~\forall u\in X_s,~~\forall s,~d\geq0,\nonumber\\
\label{E3-1R2}
&&\|(I-\Pi^{(N_i)})u\|_{s}\leq N_i^{-d}\|u\|_{s+d},~~\forall u\in X_{s+d},~~\forall s,~d\geq0.
\end{eqnarray}
Moreover, there holds
\begin{eqnarray*}
&&\Pi^{(N_i)}u:=\sum_{j\in J_{N_i}^+}\Pi_{\mathcal{N}_j}u.
\end{eqnarray*}
Consider
\begin{eqnarray}\label{E3-1}
L_au=\varepsilon f(x,u),~~where~~L_a:=-\triangle+1+\varepsilon a\triangle^{\varrho}.
\end{eqnarray}
The linearized operator of (\ref{E3-1}) has the following form
\begin{eqnarray}\label{E3-3}
L^{(N_i)}_{a}:=\Pi^{(N_i)}(L_a-\varepsilon D_{u}f(\delta,u))|_{\textbf{H}_s^{(N_i)}}.
\end{eqnarray}
Before constructing first step approximation, the invertible property of operators $L^{(N_i)}_{a}$ is needed. We give the proof of the following result in next section.
\begin{lemma}
Assume that
\begin{eqnarray}\label{E5-1R}
|m-a n|\geq\frac{\gamma_1}{\max(1,|m|^{\frac{3}{2}})},~0<\gamma_1<1,~\forall (m,n)\in\textbf{Z}^2\backslash\{(0,0)\},
\end{eqnarray}
and $\|q\|_{\bar{\sigma}}\leq 1$, $\forall$ $1\leq r\leq N$, $\forall$ $\kappa\geq1$,
\begin{eqnarray}\label{E8-1}
\|(L_a^{(r)}(\delta,q))^{-1}\|_0\leq\frac{4r^{\kappa}}{\gamma_1}.
\end{eqnarray}
Then the linearized operator $L^{(N)}_{a}(\delta,q)$ is invertible and
$\forall$ $s_2>s_1>\bar{\sigma}>0$, the linearized operator $L^{(N)}_{a}$ satisfies
\begin{eqnarray}\label{E3-4}
\|(L^{{(N)}}_{a}(\delta,q))^{-1}u\|_{s_1}\leq C(s_2-s_1)N^{\tau+\kappa_0}\left(1+\varepsilon\varsigma^{-1}\|q\|^p_{s_2}\right)^3\|u\|_{s_2},
\end{eqnarray}
where $C(s_2-s_1)=c(s_2-s_1)^{-\tau}$, $c=c(\varsigma,\tau,s,\tilde{s},\gamma_1,\gamma)$ denotes a constant.
\end{lemma}
In fact, in the iteration process, $N$ depends on the iteration step $i$.
By (\ref{E3-1}), we define
\begin{eqnarray}\label{E3-6}
\mathcal{J}_1(u)=L_{a}u-\varepsilon \Pi^{(N_i)}f(x,u)=0.
\end{eqnarray}
Next we construct the ``first step approximation".
\begin{lemma}
Assume that $a$ is diophantine. Then system
(\ref{E3-6}) has the ``first step approximation"
$u_1\in \textbf{H}_s^{(N_1)}$:
\begin{eqnarray}\label{E3-8}
u_1&=&-(L^{(N_1)}_{a})^{-1}E_0\in \textbf{H}_s^{(N_1)},\\
\label{E3-10}
E_1&=&R_0=-\varepsilon\Pi^{(N_1)}\left(f(x,u_0+u_1)-f(x,u_0)-D_uf(x,u_0)u_1\right).
\end{eqnarray}
\end{lemma}
\begin{proof}
Assume that we have chosen suitable the ``$0th$ step" approximation
solution $u_0$. Then, the target is to get the ``$1th$ step"
approximation solution.

Denote
\begin{eqnarray}\label{E3-12R}
E_0=L_{a}u_0-\varepsilon\Pi^{(N_1)}f(x,u_0).
\end{eqnarray}
By (\ref{E3-6}), we have
\begin{eqnarray}\label{E3-12}
\mathcal{J}_1(u_0+u_1)&=&L_{a}(u_0+u_1)-\varepsilon\Pi^{(N_1)}f(x,u_0+u_1)\nonumber\\
&=&L_{a}u_0-\varepsilon\Pi^{(N_1)}f(x,u_0)+L_{a}u_1+\varepsilon\Pi^{(N_1)} D_uf(x,u_0)u_1\nonumber\\
&&-\varepsilon\Pi^{(N_1)}(f(x,u_0+u_1)-f(x,u_0)-D_uf(x,u_0)u_1)\nonumber\\
&=&E_0+L^{(N_1)}_{a}u_1+R_0.
\end{eqnarray}
Then taking
\begin{eqnarray*}
E_0+L^{(N_1)}_{a}u_1=0,
\end{eqnarray*}
yields
\begin{eqnarray*}
u_1=-(L^{(N_1)}_{a})^{-1}E_0\in\textbf{H}_s^{(N_1)}.
\end{eqnarray*}
By (\ref{E3-12}), we denote
\begin{eqnarray*}
E_1&:=&R_0=\mathcal{J}_1(u_0+u_1)\nonumber\\
&=&-\varepsilon\Pi^{(N_1)}(f(x,u_0+u_1)-f(x,u_0)-D_uf(x,u_0)u_1).
\end{eqnarray*}
On the other hand, by (\ref{E3-6}) and (\ref{E3-12R}), we can obtain
\begin{eqnarray}\label{E3-12R1}
E_0=-\varepsilon(I-\Pi^{(N_0)})\Pi^{(N_1)}f(x,u_0).
\end{eqnarray}
This completes the proof.
\end{proof}

In order to prove the convergence of the Newton algorithm,  the
following KAM-style estimate is needed. For convenience, we define
\begin{eqnarray}\label{E3-12R2}
\tilde{E}_0:=-\varepsilon\Pi^{(N_1)}f(x,u_0).
\end{eqnarray}
\begin{lemma}
Assume that $a$ is diophantine. Then for any $0<\alpha<\sigma$, the
following estimates hold:
\begin{eqnarray}
&&\|u_1\|_{\sigma-\alpha}\leq C(\alpha)(1+\varepsilon\varsigma^{-1}\|u_0\|_{\sigma}^p)^{3}\|\tilde{E}_0\|_{\sigma+\tau+\kappa_0},\nonumber\\
\label{E3-13}
&&\|E_1\|_{\sigma-\alpha}\leq C^{p}(\alpha)(1+\varepsilon\varsigma^{-1}\|u_0\|_{\sigma}^p)^{3p}\|\tilde{E}_0\|^{p}_{\sigma+\tau+\kappa_0},
\end{eqnarray}
where $C(\alpha)$ is defined in (\ref{E2-12'}).
\end{lemma}
\begin{proof}
Denote
\begin{eqnarray}\label{E2-12'}
C(\alpha)&=&c(\varsigma,\tau,s,\tilde{s},\gamma_1,\gamma)\alpha^{-\tau}.
\end{eqnarray}
From the definition of $u_1$ in
(\ref{E3-8}), by Lemma 12, (\ref{E3-1R2}) and (\ref{E3-12R2}), we derive
\begin{eqnarray}\label{E3-14}
\|u_1\|_{\sigma-\alpha}&=&\|-(L^{( N_1)}_{a})^{-1}E_0\|_{\sigma-\alpha}\nonumber\\
&\leq&C(\alpha)N_1^{\tau+\kappa_0}(1+\varepsilon\varsigma^{-1}\|u_0\|_{\sigma}^p)^3\|E_0\|_{\sigma}\nonumber\\
&\leq&C(\alpha)(1+\varepsilon\varsigma^{-1}\|u_0\|_{\sigma}^p)^3\|\tilde{E}_0\|_{\sigma+\tau+\kappa_0}.
\end{eqnarray}
By assumption (\ref{E1-8}) and the
definition of $E_1$, we have
\begin{eqnarray*}
\|E_1\|_{\sigma-\alpha}
&=&\|\Pi^{(N_1)}(f(x,u_0+u_1)-f(x,u_0)-D_uf(x,u_0)u_1)\|_{\sigma-\alpha}\nonumber\\
&\leq&\|u_1\|_{\sigma-\alpha}^{p}\nonumber\\
&\leq&C^{p}(\alpha)(1+\varepsilon\varsigma^{-1}\|u_0\|_{\sigma}^p)^{3p}\|\tilde{E}_0\|^p_{\sigma+\tau+\kappa_0}.
\end{eqnarray*}
This completes the proof.
\end{proof}
For $i\in\textbf{N}$ and $0<\sigma_{0}(\mathcal{M})<\bar{\sigma}(\mathcal{M})<\sigma(\mathcal{M})<k(\mathcal{M})-1$, set
\begin{eqnarray}\label{E3-16}
&&\sigma_i:=\bar{\sigma}+\frac{\sigma-\bar{\sigma}}{2^i},\\
\label{E3-17}
&&\alpha_{i+1}:=\sigma_i-\sigma_{i+1}=\frac{\sigma-\bar{\sigma}}{2^{i+1}}.
\end{eqnarray}
By (\ref{E3-16})-(\ref{E3-17}), it follows that
\begin{eqnarray*}
\sigma_0>\sigma_1>\ldots>\sigma_i>\sigma_{i+1}>\ldots,~for~i\in\textbf{N}.
\end{eqnarray*}
Define
\begin{eqnarray*}
&&\mathcal{P}_1(u_0):=u_0+u_1,~~for~u_0\in \textbf{H}_{\sigma_0}^{(N_0)},\\
&&E_{i}=\mathcal{J}_1(\sum_{k=0}^iu_k)=\mathcal{J}_1(\mathcal{P}_1^i(u_0)),
\end{eqnarray*}
In fact, to obtain the ``$i$ th" approximation solution
$u_i\in\textbf{H}_{\sigma_i}^{(N_i)}$ of system
(\ref{E3-6}), we need to solve following equations
\begin{eqnarray*}
\mathcal{J}_1(\sum_{k=0}^iu_k)&=&L_{a}(\sum_{k=0}^{i-1}
u_k)-\varepsilon\Pi^{(N_i)}f(x,\sum_{k=0}^{i-1}u_k)+L_{a}u_i-\varepsilon\Pi^{(N_i)}D_uf(x,\sum_{k=0}^{i-1}u_k)u_i\nonumber\\
&&-\varepsilon\Pi^{(N_i)}\left(f(x,\sum_{k=0}^{i}u_k)- f(x,\sum_{k=0}^{i-1}u_k)-D_uf(x,\sum_{k=0}^{i-1}u_k)u_i\right).
\end{eqnarray*}
Then, we get the `$i$ th` step approximation $u_i\in
\textbf{H}_{\sigma_i}^{(N_i)}$ :
\begin{eqnarray}\label{E2-32}
u_i=-(L_{a}^{(N_i)})^{-1}E_{i-1},
\end{eqnarray}
where
\begin{eqnarray*}
E_{i-1}=L_{a}(\sum_{k=0}^{i-1}
u_k)-\varepsilon\Pi^{(N_i)}f(x,\sum_{k=0}^{i-1}u_k)=-\varepsilon(I-\Pi^{(N_{i-1})})\Pi^{(N_i)}f(x,\sum_{k=0}^{i-1}u_k).
\end{eqnarray*}
As done in Lemma 13, it is easy to get that
\begin{eqnarray}\label{E2-32r}
E_i:=R_{i-1}=-\varepsilon\Pi^{(N_i)}\left(f(x,\sum_{k=0}^{i-1}u_k)-f(x,\sum_{k=0}^{i}u_k)-D_uf(x,\sum_{k=0}^{i-1}u_k)u_i\right),~~~
\end{eqnarray}
\begin{eqnarray}
\label{E2-32rr}
\tilde{E}_i=-\varepsilon\Pi^{(N_i)}f(x,\sum_{k=0}^{i-1}u_k).
\end{eqnarray}
Hence, we only need to estimate $R_{i-1}$
to prove the convergence of algorithm. In the following, a
sufficient condition on the convergence of Newton algorithm is
proved. This proof is based on Lemma 14. It also shows the existence of solutions for
(\ref{E3-6}).
\begin{lemma}
Assume that $a$ is diophantine. Then, for sufficiently small $\varepsilon$, equations
(\ref{E3-1}) has a solution
\begin{eqnarray*}
u_{\infty}=\sum_{k=0}^{\infty}u_k\in
\textbf{H}_{\bar{\sigma}}\cap\mathcal{B}_1(0),
\end{eqnarray*}
where $\mathcal{B}_1(0):=\{u|\|u\|_s\leq1,~\forall s>\bar{s}>0\}.$
\end{lemma}
\begin{proof}
We divide into two cases. If $\varepsilon\varsigma^{-1}\|u_{i-1}\|_{\sigma_{i-1}}^p<1$,
by Lemma 12, (\ref{E2-32}) and (\ref{E2-32rr}), we derive
\begin{eqnarray}\label{E3-18}
\|u_i\|_{\sigma_i}&=&\|-(L_{a}^{(N_i)})^{-1}E_{i-1}\|_{\sigma_i}\nonumber\\
&\leq&C(\alpha_i)N_{i}^{\tau+\kappa_0}(1+\varepsilon\varsigma^{-1}\|u_{i-1}\|_{\sigma_{i-1}}^p)^3\|E_{i-1}\|_{\sigma_{i-1}}\nonumber\\
&\leq&C(\alpha_i)(1+\varepsilon\varsigma^{-1}\|u_{i-1}\|_{\sigma_{i-1}}^{p})^3\|\tilde{E}_{i-1}\|_{\sigma_{i-1}+\tau+\kappa_0}\nonumber\\
&\leq&2C(\alpha_i)\|\tilde{E}_{i-1}\|_{\sigma_{i-1}+\tau+\kappa_0},
\end{eqnarray}
where $c(\varepsilon,\varsigma)$ is a constant depending on $\varepsilon$ and $\varsigma$.

Note that $N_i=N_0^i$, $\forall i\in\textbf{N}$. By (\ref{E2-32r})-(\ref{E3-18}) and
assumption (\ref{E1-8}), we have
\begin{eqnarray}\label{E3-20}
\|E_i\|_{\sigma_i}&=&\varepsilon\|\Pi^{(N_i)}(f(x,\sum_{k=0}^{i}u_k)-f(x,\sum_{k=0}^{i-1}u_k)-D_uf(x,\sum_{k=0}^{i-1}u_k)u_i)\|_{\sigma_i}\nonumber\\
&\leq&\varepsilon c(s)\|u_i\|_{\sigma_i}^p\nonumber\\
&\leq&\varepsilon c(s)N_i^{(\tau+\kappa_0)p}C^p(\alpha_i)\|E_{i-1}\|_{\sigma_{i-1}}^p\nonumber\\
&\leq&(\varepsilon c(s))^{p+1}N_i^{(\tau+\kappa_0)p}N_{i-1}^{(\tau+\kappa_0)p^2}C^p(\alpha_i)C^{p^2}(\alpha_{i-1})\|E_{i-2}\|_{\sigma_{i-2}}^{p^2}\nonumber\\
&\leq&\cdots\nonumber\\
&\leq&(\varepsilon c(s))^{\sum_{k=1}^{i-1}p^k+1}N_0^{(\tau+\kappa_0)p^{i+2}}\|E_{0}\|_{\sigma_{0}}^{p^i}
\prod_{k=1}^{i}C^{p^k}(\alpha_{i+1-k})\nonumber\\
&\leq&(\varepsilon c(s))^{p^i}(\varepsilon,\varsigma)(N_0^{(\tau+\kappa_0)p^2}\|E_{0}\|_{\sigma_{0}})^{p^i}
\prod_{k=1}^{i}C^{p^k}(\alpha_{i+1-k})\nonumber\\
&\leq&(\varepsilon c(s))^{p^i}(\varepsilon,\varsigma)\|\tilde{E}_{0}\|_{\sigma_{0}+(\tau+\kappa_0)p^2}^{p^i}
\prod_{k=1}^{i}C^{p^k}(\alpha_{i+1-k})\nonumber\\
&\leq&(8^{p^2}\varepsilon c(s)c^{p^2}(\tau,\sigma,\tilde{\sigma},\gamma_1,\gamma)\|\tilde{E}_0\|_{\sigma_{0}+(\tau+\kappa_0)p^2})^{p^i}.
\end{eqnarray}
Hence, choosing small $\varepsilon>0$ such that
\begin{eqnarray*}
8^{p^2}\varepsilon c(s)c^{p^2}(\tau,\sigma,\tilde{\sigma},\gamma_1,\gamma)\|\tilde{E}_0\|_{\sigma_{0}+(\tau+\kappa_0)p^2}=8^{p^2}\varepsilon c(s)c^{p^2}(\tau,\sigma,\tilde{\sigma},\gamma_1,\gamma)N_0^{(\tau+\kappa_0)p^2}\|\tilde{E}_0\|_{\sigma_{0}}<1.
\end{eqnarray*}
For any fixed $p>1$, we derive
\begin{eqnarray}\label{E3-22R}
\lim_{i\longrightarrow\infty}\|E_i\|_{\sigma_i}=0.
\end{eqnarray}
If $\varepsilon\varsigma^{-1}\|u_{i-1}\|_{\sigma_{i-1}}^p\geq1$,
by Lemma 12, (\ref{E2-32}) and (\ref{E2-32rr}), we derive
\begin{eqnarray}\label{E3-18R}
\|u_i\|_{\sigma_i}&=&\|-(L_{a}^{(N_i)})^{-1}E_{i-1}\|_{\sigma_i}\nonumber\\
&\leq&C(\alpha_i)N_i^{\tau+\kappa_0}(1+\varepsilon\varsigma^{-1}\|u_{i-1}\|_{\sigma_{i-1}}^p)^3\|E_{i-1}\|_{\sigma_{i-1}}\nonumber\\
&\leq&2\varepsilon^3\varsigma^{-3}C(\alpha_i)\|u_{i-1}\|_{\sigma_{i-1}}^{3p}\|\tilde{E}_{i-1}\|_{\sigma_{i-1}+\tau+\kappa_0}\nonumber\\
&\leq&(2\varepsilon\varsigma^{-1})^{3(p+1)}C(\alpha_{i})C^{3p}(\alpha_{i-1})\|u_{i-2}\|_{\sigma_{i-2}}^{(3p)^2}\|\tilde{E}_{i-2}\|^{3p}_{\sigma_{i-2}+\tau+\kappa_0}\|\tilde{E}_{i-1}\|_{\sigma_{i-1}+\tau+\kappa_0}\nonumber\\
&\leq&\cdots\nonumber\\
&\leq&(2\varepsilon\varsigma^{-1})^{\sum_{k=0}^{i-1}(3p)^k}\|u_0\|^{(3p)^i}_{\sigma_0}\prod_{k=1}^iC^{(3p)^{k-1}}(\alpha_{i+1-k})\|\tilde{E}_{i-k}\|^{(3p)^{k-1}}_{\sigma_{i-k}+\tau+\kappa_0}.
\end{eqnarray}
But we will choose the initial step $u_0=0$ in this paper, which combining with (\ref{E3-18R}) leads to $\|u_i\|_{\sigma_i}=0$, $\forall i\in\textbf{N}$. This contradicts with assumption $\varepsilon\varsigma^{-1}\|u_{i-1}\|_{\sigma_{i-1}}^p>1$. Hence, the case is not possible.
(\ref{E3-1}) has a solution
\begin{eqnarray*}
u_{\infty}:=\sum_{k=0}^{\infty}u_k\in
\textbf{H}_{\bar{\sigma}}\cap\mathcal{B}_1(0),
\end{eqnarray*}
where $\mathcal{B}_1(0):=\{u|\|u\|_s\leq1,~\forall s>\bar{s}>0\}.$
This completes the proof.
\end{proof}
Next result gives the local uniqueness of solutions for equation (\ref{E3-1}).
\begin{lemma}
Assume that $a$ is diophantine. Equation (\ref{E3-1})
has a unique solution $u\in
\textbf{H}_{\bar{\sigma}}\cap\textbf{B}_1(0)$ obtained in
Lemma 15.
\end{lemma}
\begin{proof}
Let $u,\tilde{u}\in
\textbf{H}_{\bar{\sigma}}\cap\textbf{B}_1(0)$ be two solutions of
system (\ref{E3-6}), where
\begin{eqnarray*}
\textbf{B}_1(0):=\{u|\|u\|_s<\delta,~for~some~\delta<1,~\forall s>\sigma_{0}\}.
\end{eqnarray*}
Write $h=u-\tilde{u}$. Our target is to prove $h=0$.
By (\ref{E3-6}), we have
\begin{eqnarray*}
L_{a}h-\varepsilon\Pi^{(N_i)}D_uf(x,u)h-\varepsilon\Pi^{(N_i)}
(f(x,u)-f(x,\tilde{u})-D_uf(x,u)h=0,
\end{eqnarray*}
which implies that
\begin{eqnarray}
\label{E3-24}
h=\varepsilon(L_{a}-\varepsilon\Pi^{(N_i)}D_uf(x,u))^{-1}\Pi^{(N_i)}(f(x,u)-f(x,\tilde{u})-D_uf(x,u)h).
\end{eqnarray}
Note that $N_i=N_0^i$, $\forall i\in\textbf{N}$. Thus, by Lemma 12 and (\ref{E3-24}), we have
\begin{eqnarray*}
\|h\|_{\sigma_i}&=&\varepsilon\|(L^{N_i}_{a})^{-1}\Pi^{(N_i)}(f(x,u)-f(x,\tilde{u})-D_uf(x,u)h)\|_{\sigma_i}\nonumber\\
&\leq&C(\alpha_i)N_i^{\tau+\kappa_0}(1+\varepsilon\varsigma^{-1}\|u\|_{\sigma_{i-1}}^p)\|h\|^p_{\sigma_{i-1}}\nonumber\\
&\leq&2^{p+1}N_i^{(\tau+\kappa_0)}N_{i-1}^{(\tau+\kappa_0)p}C(\alpha_{i})C^{p}(\alpha_{i-1})\|h\|_{\sigma_{i-2}}^{p}\nonumber\\
&\leq&\cdots\nonumber\\
&\leq&2^{\sum_{k=0}^{i-1}p^k}N_0^{(\tau+\kappa_0)(\sum_{k=0}^{i-1}p^k)}\|h\|^{p^i}_{\sigma_0}\prod_{k=1}^iC^{p^{k-1}}(\alpha_{i+1-k})\nonumber\\
&\leq&(8^{p^2}c^{p^2}(\varepsilon,\varsigma,\tau,s,\tilde{s},\gamma_1,\gamma)N_0^{(\tau+\kappa_0)p}\|h\|_{\sigma_{0}})^{p^i}.
\end{eqnarray*}
Choosing $\delta<8^{-p^2}c^{-p^2}(\varepsilon,\varsigma,\tau,s,\tilde{s},\gamma_1,\gamma)N_0^{-(\tau+\kappa_0)p}$, we obtain
\begin{eqnarray*}
\lim_{i\longrightarrow\infty}\|h\|_{\bar{\sigma}}=0.
\end{eqnarray*}
This completes the proof.
\end{proof}
\begin{remark}
The dependence upon the parameter, as is well known, is more delicated
since it involves in the small divisors of $\omega_j$: it is, however, standard to check
that this dependence is $\textbf{C}^1$ on a bounded set of
Diophantine numbers, for more details, see, for example, \cite{Berti2,Berti1}.
\end{remark}

By Lemma 12, for sufficient small $\delta_0>0$ and given $r>0$, we define
\begin{eqnarray*}
&&Y_{\gamma_1,\kappa_0}^{(N)}:=\{(\delta,q)\in[0,\delta_0)\times\textbf{H}^{(N)}|\|q\|_{\bar{\sigma}}\leq 1,\varepsilon{\delta}~satisfies~(\ref{E5-1R})-(\ref{E8-1})\},\\
&&U_r^{(N)}:=\{u\in\textbf{C}^1([0,\delta_0),\textbf{H}^{N})|\|u\|_{\bar{\sigma}}\leq 1,~\|\partial_{\delta}u\|_{\bar{\sigma}}\leq r\},\\
&&\mathcal{G}_{\gamma_1,\kappa_0}^{(N)}:=\{\delta\in[0,\delta_0)|(\delta,u(\delta))\in Y_{\gamma_1,\kappa_0}^{(N)}~and~u\in U_r^{(N)}\},\\
&&\mathcal{G}_r:=\{\delta\in[0,\delta_0)|\|(L_a^{(r)}(\delta,q(\delta)))^{-1}\|\leq \frac{4r^{\kappa}}{\gamma_1}\},\\
&&\mathcal{G}:=\{\delta\in[0,\delta_0)|\omega(\delta)~satisfies~(\ref{E5-1R})\}.
\end{eqnarray*}
Then for a given function $\delta\mapsto q(\delta)\in U_r^{(N)}$, the set $\mathcal{G}_{\gamma_1,\kappa_0}^{(N)}$ is equivalent to
\begin{eqnarray*}
\mathcal{G}_{\gamma_1,\kappa_0}^{(N)}=\cap_{1\leq r\leq N}\mathcal{G}_r\cap\mathcal{G}.
\end{eqnarray*}
Choosing $\kappa$ and $\gamma_1$ such that
\begin{eqnarray}\label{E8-2}
\kappa\geq\max\{\tau,2+d+n+\frac{2\rho-2}{2\rho-1}(\tau+2\varrho)\},~~\gamma_1\in(0,\gamma_2],~for~\gamma_2\leq\gamma_1.
\end{eqnarray}
Next we have the measure estimate. The proof of it will be given in Appendix.
\begin{lemma}(Measure estimates)
Assume that $\varepsilon$ is diophantine, $\varepsilon_0\gamma^{-1}M^{\tau+2\varrho}$ is sufficient small and (\ref{E8-2}) holds. Then $\mathcal{G}_{\gamma_1,\kappa_0}^{(M)}(0)=\mathcal{G}$, and $\mathcal{G}$ satisfies
\begin{eqnarray}\label{E8-3}
|(\mathcal{G}_{\gamma_1,\kappa_0}^{(M)}(0))^c\cap[0,\delta)|\leq C\gamma_1\delta,~~\forall \delta\in(0,\delta_0].
\end{eqnarray}
Furthermore, for any $r'>0$, there exists $\delta':=\delta'(\gamma_1,r')$ such that the measure estimate
\begin{eqnarray}\label{E8-4}
|(\mathcal{G}_{\gamma_1,\kappa_0}^{(N')}(u_2))^c\backslash(\mathcal{G}_{\gamma_1,\kappa_0}^{(N)}(u_1))^c\cap[0,\delta)|\leq C\gamma_1\delta N^{-1},~~\forall \delta\in(0,\delta']
\end{eqnarray}
holds, where $N'\geq N\geq M$, $u_1\in U_{r'}^{(N)}$, $u_1\in U_{r'}^{(N)}$ with $\|u_2-u_1\|_{\bar{\sigma}}\leq N^{-e}$, $e$ denotes a constant depending on $\kappa_0$ and $n$.
\end{lemma}

\section{Proof of Lemma 12}
This section is devoted to give the proof of Lemma 12.  Let
\begin{eqnarray*}
b(x):=(\partial_uf)(\delta,u).
\end{eqnarray*}
For notational convenience, we denote $N=N_i$. Due to the orthogonal decomposition $\textbf{H}^{(N)}=\bigoplus_{j\in J_N^+}\mathcal{N}_j$, we define
\begin{eqnarray}\label{E4-1}
h\mapsto\textbf{L}^{(N)}[h]:=\Pi^{(N)}(L_ah-\varepsilon b(x)h),~~\forall h\in\textbf{H}^{(N)}.
\end{eqnarray}
We write (\ref{E4-1}) by the block matrix
\begin{eqnarray}\label{E4-3}
L^{(N)}_a=D+\varepsilon T,~~D:=diag_{j\in J^+_N}(D_jI_j),
\end{eqnarray}
where $j=(j_1,j_2)\in\Lambda^+\times\textbf{Z}^n$,
\begin{eqnarray}\label{E4-5}
D_j:=|j_1+\rho|^2-|\rho|^2+|j_2|^2+1-\varepsilon a(|j_1+\rho|^2-|\rho|^2+|j_2|^2)^{\varrho},
\end{eqnarray}
$I_j$ is the identity map in $\mathcal{N}_j$, and
\begin{eqnarray}\label{E4-7}
T:=(T^{j'}_j)_{j,j'\in J^+_N},~~T^{j'}_j:=\Pi_{\mathcal{N}_j}L_a^{(N)}|_{\mathcal{N}_{j'}}\in\mathcal{L}(\mathcal{N}_{j'},\mathcal{N}_{j}).
\end{eqnarray}
In what follows, we prove the estimate (\ref{E3-4}).
For fixing $\varsigma>0$, we define the regular sites $R$ and the singular sites $S$ as
\begin{eqnarray}\label{E4-9}
R:=\{j\in J^+||D_j|\geq\varsigma\}~~and~~S:=\{j\in J^+||D_j|<\varsigma\}.
\end{eqnarray}
For each $N$, we denote the restrictions of $S$, $R$, $\Omega_{\alpha}$ to $J_N^+$ with the same symbols.
The following result shows the separation of singular sites, and the proof will be completed in the appendix.
\begin{lemma}
Assume that $a$ is diophantine and $a$ satisfies (\ref{E5-1R}). There exists $\varsigma_0(\gamma)$ such that
for $\varsigma\in(0,\varsigma_0(\gamma)]$ and a partition of the singular sites $S$ which can be partitioned in pairwise disjoint clusters $\Omega_{\alpha}$ as
\begin{eqnarray}\label{E4-10}
S=\bigcup_{\alpha\in\textbf{N}}\Omega_{\alpha}
\end{eqnarray}
satisfying

$\bullet$ (dyadic) $\forall\alpha$, $M_{\alpha}\leq2m_{\alpha}$, where $M_{\alpha}:=\max_{j\in\Omega_{\alpha}}|j+\overrightarrow{\rho}|$, $m_{\alpha}:=\max_{j\in\Omega_{\alpha}}|j+\overrightarrow{\rho}|$.

$\bullet$ (separation) $\exists \lambda, c>0$ such that $d(\Omega_{\alpha},\Omega_{\beta})\geq c(M_{\alpha}+M_{\beta})^{\lambda}$, $\forall\alpha\neq\beta$, where $d(\Omega_{\alpha},\Omega_{\beta}):=\max_{j\in\Omega_{\alpha},j'\in\Omega_{\beta}}|j-j'|$ and $\lambda$ depends only on $\mathcal{M}$.
\end{lemma}
Using Lemma 10, we have the following.
\begin{lemma}
Let $2s'\geq d+r+2n+3$. For a real $b\in\textbf{H}_{s+s'}$, the matrix $T=(T_j^{j'})_{j,j'\in J^+_N}$ defined in (\ref{E4-7}) is self-adjoint and belongs to the algebra of polynomially localized matrices $\mathcal{A}_{s}(J_N^+)$ with
\begin{eqnarray*}
|T|_s\leq K(s)\|b\|_{s+s'}.
\end{eqnarray*}
Moreover, for any $s>s'$,
\begin{eqnarray*}
|T|_s\leq K'(s)N^{s'}\|b\|_{s}.
\end{eqnarray*}
\end{lemma}
Since the decomposition
\begin{eqnarray*}
\textbf{H}^{(N)}:=\textbf{H}_R\oplus\textbf{H}_S,
\end{eqnarray*}
with
\begin{eqnarray*}
\textbf{H}_R:=\bigoplus_{j\in R\cap J^+_N}\mathcal{N}_j,~~\textbf{H}_S:=\bigoplus_{j\in S\cap J^+_N}\mathcal{N}_j,
\end{eqnarray*}
we can represent the operator $L^{(N)}_a$ as the self-adjoint block matrix
\begin{eqnarray*}
L^{(N)}_a=\left(
\begin{array}{ccc}
L_R& L_R^S\\
L^R_S&L_S
\end{array}
\right),
\end{eqnarray*}
where $L_R^S=(L_S^R)^{\dag}$, $L_R=L_R^{\dag}$, $L_S=L_S^{\dag}$.

Thus the invertibility of $L^{(N)}_a$ can be expressed via the ''resolvent-type'' identity
\begin{eqnarray}\label{E4-22}
(L^{(N)}_a)^{-1}=\left(
\begin{array}{ccc}
I&-L_R^{-1}L_R^S\\
0&I
\end{array}
\right)
\left(
\begin{array}{ccc}
L_R^{-1}&0\\
0&\mathcal{L}^{-1}
\end{array}
\right)
\left(
\begin{array}{ccc}
I&0\\
-L_S^RL_R^{-1}&I
\end{array}
\right),
\end{eqnarray}
where the ''quasi-singular'' matrix
\begin{eqnarray*}
\mathcal{L}:=L_S-L_S^RL_R^{-1}L_R^S\in\mathcal{A}_s(S).
\end{eqnarray*}
The reason of $\mathcal{L}\in\mathcal{A}_s(S)$ is that $\mathcal{L}$ is the restriction to $S$ of the polynomially localized matrix
\begin{eqnarray*}
I_S(L-I_SLI_R\tilde{L}^{-1}I_RLI_S)I_S\in\mathcal{A}_s,
\end{eqnarray*}
where
\begin{eqnarray*}
\tilde{L}^{-1}=\left(
\begin{array}{ccc}
I&0\\
0&L_R
\end{array}
\right).
\end{eqnarray*}
\begin{lemma}
Assume that nonresonance condition (\ref{E1-5}) holds. For $s_0<s_1<s_2<k-1$, $|L_R^{-1}|_{s_0}\leq2\varsigma^{-1}$, the operator $L_R$ satisfies
\begin{eqnarray}\label{E4-11}
&&|\tilde{L}_R^{-1}|_{s_1}\leq c(s_1)(1+\varepsilon\varsigma^{-1}|T|_{s_1}),\\
\label{E4-12}
&&\|L_R^{-1}u\|_{s_1}\leq c(\gamma,\tau,s_2)(s_2-s_1)^{-\tau}(1+\varepsilon\varsigma^{-1}|T|_{s_2})\|u\|_{s_2},
\end{eqnarray}
where $\tilde{L}^{-1}=L^{-1}_RD_R$, $c(\gamma,\tau,s_2)$ is a constant depending on $\gamma,\tau,s_2$.
\end{lemma}
\begin{proof}
It follows from (\ref{E4-3}) and (\ref{E4-9}) that $D_R$ is a diagonal matrix and satisfies $|D_R^{-1}|_s\leq\varsigma^{-1}$. By (\ref{E2-4}), we have that the Neumann series
\begin{eqnarray}\label{E4-13}
\tilde{L}_R^{-1}=L_R^{-1}D_R=\sum_{m\geq0}(-\varepsilon)^m(D_R^{-1}T_R)^m
\end{eqnarray}
is totally convergent in $|\cdot|_{s_1}$ with $|L_R^{-1}|_{s_0}\leq2\varsigma^{-1}$, by taking $\varepsilon\varsigma^{-1}|T|_{s_0}\leq c(s_0)$ small enough.

Using (\ref{E2-4}) and (\ref{E2-8}), we have that $\forall m\in\textbf{N}$,
\begin{eqnarray*}
\varepsilon^m|(D_R^{-1}T_R)^m|_{s_1}&\leq&\varepsilon^mc(s)|(D_R^{-1}T_R)^m|_{s_1}\\
&\leq&c(s)\varepsilon^mm(c(s)|D_R^{-1}T_R|_{s_0})^{m-1}|D_R^{-1}T_R|_{s_1}\\
&\leq&c'(s)\varepsilon m\varsigma^{-1}(\varepsilon c(s_1)\varsigma^{-1}|T|_{s_0})^{m-1}|T|_{s_1},
\end{eqnarray*}
which together with (\ref{E4-13}) implies that for $\varepsilon\varsigma^{-1}|T|_{s_0}< c(s_0)$ small enough, (\ref{E4-11}) holds.

By nonresonance condition (\ref{E1-5}) and $\sup_{x>0}(x^ye^{-x})=(ye^{-1})^y$, $\forall y\geq0$, we derive
\begin{eqnarray}\label{E4-14}
e^{-2|j+\overrightarrow{\rho}|(s_2-s_1)}|\omega_j^2+1-\varepsilon a\omega_j^{2p}|^{-2}&\leq&\gamma^{-1}|j+\overrightarrow{\rho}|^{\tau}e^{-2|j+\overrightarrow{\rho}|(s_2-s_1)}\nonumber\\
&\leq&c(\gamma,\tau)(s_2-s_1)^{-2\tau}.
\end{eqnarray}
Then by (\ref{E4-14}), for any $u\in\textbf{H}_R$,
\begin{eqnarray*}
\|L_R^{-1}u\|_{s_1}^2&=&\sum_{j\in R\cap J_N^+}e^{2|j+\overrightarrow{\rho}|s_1}\|L_R^{-1}u_j\|_{\textbf{L}^2}^2\\
&\leq&\sum_{j\in R\cap J_N^+}e^{2|j+\overrightarrow{\rho}|s_1}|\omega_j^2+1-\varepsilon a\omega_j^{2p}|^{-2}\|\tilde{L}_R^{-1}u_j\|_{\textbf{L}^2}^2\\
&\leq&\sum_{j\in R\cap J_N^+}e^{-2|j+\overrightarrow{\rho}|(s_2-s_1)}|\omega_j^2+1-\varepsilon a\omega_j^{2p}|^{-2}e^{2|j+\overrightarrow{\rho}|s_2}\|\tilde{L}_R^{-1}u_j\|_{\textbf{L}^2}^2\\
&\leq&c(\gamma,\tau)(s_2-s_1)^{-2\tau}\|\tilde{L}_R^{-1}u\|_{s_2}^2.
\end{eqnarray*}
Thus using interpolation (\ref{E2-6}) and (\ref{E4-11}), we derive that for $s_1<s<s_2$,
\begin{eqnarray*}
\|L_R^{-1}u\|_{s_1}
&\leq&c(\gamma,\tau)(s_2-s_1)^{-\tau}\|\tilde{L}_R^{-1}u\|_{s_2}\\
&\leq&c(r,\tau,s_2)(s_2-s_1)^{\tau}(|\tilde{L}_R^{-1}|_{s_2}\|u\|_{s}+|\tilde{L}_R^{-1}|_{s}\|u\|_{s_2})\\
&\leq&c(r,\tau,s_2)(s_2-s_1)^{\tau}(1+\varepsilon\varsigma^{-1}|T|_{s_2})\|u\|_{s_2}.
\end{eqnarray*}
This completes the proof.
\end{proof}
Next we analyse the quasi-singular matrix $\mathcal{L}$. By (\ref{E4-10}), the singular sites restricted to $J_N^+$ are
\begin{eqnarray*}
S=\bigcup_{\alpha\in l_N}\Omega_{\alpha},~~where~l_N:=\{\alpha\in\textbf{N}|m_{\alpha}\leq N\},
\end{eqnarray*}
and $\Omega_{\alpha}\equiv\Omega_{\alpha}\cup J_N^+$. Due to the decomposition $\tilde{H}_S:=\bigoplus_{\alpha\in l_N}\tilde{H}_{\alpha}$, where $\textbf{H}_{\alpha}:=\bigoplus_{j\in\Omega_{\alpha}}\mathcal{N}_j$, we represent $\mathcal{L}$ as the block matrix $\mathcal{L}=(\mathcal{L}_{\alpha}^{\beta})_{\alpha,\beta\in l_N}$, where $\mathcal{L}_{\alpha}^{\beta}:=\Pi_{\textbf{H}_{\alpha}}\mathcal{L}|_{\textbf{H}_{\beta}}$.
So we can rewrite
\begin{eqnarray*}
\mathcal{L}=\mathcal{D}+\mathcal{T},
\end{eqnarray*}
where $\mathcal{D}:=diag_{\alpha\in l_N}(\mathcal{L}_{\alpha})$, $\mathcal{L}_{\alpha}:=\mathcal{L}_{\alpha}^{\alpha}$, $\mathcal{T}:=(\mathcal{L}_{\alpha}^{\beta})_{\alpha\neq\beta}$.

We define a diagonal matrix corresponding to the matrix $\mathcal{D}$ as
$\bar{D}:=diag_{\alpha\in l_N}(\bar{L}_{\alpha})$, where $\bar{L}_{\alpha}=diag_{j\in\Omega_{\alpha}}(D_j)$.

To show $\mathcal{D}$ is invertible, we only need to prove that $\mathcal{L}_{\alpha}$ is invertible, $\forall\alpha\in l_N$.
\begin{lemma}
$\forall\alpha\in l_N$, $\mathcal{L}_{\alpha}$ is invertible and $\|\mathcal{L}_{\alpha}^{-1}\|_0\leq C\gamma_1^{-1}M_{\alpha}^{\kappa}$.
\end{lemma}
The proof process of above Lemma is similar with Lemma 6.6 in \cite{Berti1}, so we omit it.
\begin{lemma}
Assume that nonreonance condition (\ref{E1-5}) holds. We have
\begin{eqnarray*}
\|\mathcal{D}^{-1}\bar{D}u\|_{s_1}
\leq c(\varsigma,s_1,\gamma_1)N^{\tau}\|u\|_{s_2},
\end{eqnarray*}
where $c(\varsigma,s_1,\gamma_1)$ is a constant which depends on $\varsigma$, $s_1$ and $\gamma_1$.
\end{lemma}
\begin{proof}
Note that $\|u_{\alpha}\|_0\leq m_{\alpha}^{-s_1}\|u_{\alpha}\|_{s_1}$ and $M_{\alpha}=2m_{\alpha}$.
So for any $u=\sum_{\alpha\in l_N}u_{\alpha}\in\textbf{H}_{\alpha}$, $u_{\alpha}\in\textbf{H}_{\alpha}$,
\begin{eqnarray}\label{E4-15}
\|\mathcal{D}^{-1}\bar{D}u\|_{s_1}^2&=&\sum_{\alpha\in l_N}\|\mathcal{L}_{\alpha}^{-1}\bar{L}_{\alpha}u_{\alpha}\|_{s_1}^2\leq\sum_{\alpha\in l_N}M_{\alpha}^{2s_1}\|\mathcal{L}_{\alpha}^{-1}\bar{L}_{\alpha}u_{\alpha}\|_{0}^2\nonumber\\
&\leq&c\gamma_1^{-2}\sum_{\alpha\in l_N}M_{\alpha}^{2(s_1+\tau)}\|\bar{L}_{\alpha}u_{\alpha}\|_{0}^2\nonumber\\
&\leq&c\gamma_1^{-2}\sum_{\alpha\in l_N}M_{\alpha}^{2(s_1+\tau)}m_{\alpha}^{-2s_1}\|\bar{L}_{\alpha}u_{\alpha}\|_{s_1}^2\nonumber\\
&\leq&c\gamma_1^{-2}4^{s_1}\sum_{\alpha\in l_N}M_{\alpha}^{2\tau}\|\bar{L}_{\alpha}u_{\alpha}\|_{s_1}^2\nonumber\\
&\leq&c\gamma_1^{-2}4^{s_1}N^{2\tau}\sum_{\alpha\in l_N}\|\bar{L}_{\alpha}u_{\alpha}\|_{s_1}^2\nonumber\\
&=&c\gamma_1^{-2}4^{s_1}N^{2\tau}\|\bar{D}u\|_{s_1}^2.
\end{eqnarray}
Using interpolation (\ref{E2-6}) and (\ref{E4-9}), for $0<s_1<s_2$, it follows from (\ref{E4-15}) that
\begin{eqnarray*}
\|\mathcal{D}^{-1}\bar{D}u\|_{s_1}&\leq&c\gamma_1^{-1}2^{s_1}N^{\tau}\|\bar{D}u\|_{s_1}\\
&\leq&c\gamma_1^{-1}2^{s_1}N^{\tau}(|\bar{D}|_{s_2}\|u\|_{s_1}+|\bar{D}|_{s_1}\|u\|_{s_2})\\
&\leq&c(\varsigma)\gamma_1^{-1}2^{s_1+1}N^{\tau}\|u\|_{s_2}.
\end{eqnarray*}
This completes the proof.
\end{proof}
The following result is taken from \cite{Berti1}, so we omit the proof.
\begin{lemma}
For $\kappa_0=\tau+r+n+1$, $\forall s\geq0$, $\forall m\in\textbf{N}$, there hold:
\begin{eqnarray}\label{E4-17}
&&c(s_1)\|\mathcal{D}^{-1}\mathcal{T}\|_{s_0}<\frac{1}{2},~~\|\mathcal{D}^{-1}\|_s\leq c(s)\gamma_1^{-1}N^{\tau},\\
\label{E4-19}
&&\|(\mathcal{D}^{-1}\mathcal{T})^mu\|_s\leq(\varepsilon\gamma^{-1}K(s))^m(mN^{\kappa_0}|T|_s|T|_{s_0}^{m-1}\|u\|_{s_0}+|T|^m_{s_0}\|u\|_s).~~~~~~~
\end{eqnarray}
\end{lemma}
\begin{lemma}
Assume that nonreonance condition (\ref{E1-5}) holds. For $0<s_0<s_1<s_2<s_3<k-1$, we have
\begin{eqnarray}
\label{E4-18}
\|\mathcal{L}^{-1}u\|_{s_1}\leq c(\varsigma,\tau,s_1,\gamma_1,\gamma)N^{\tau+\kappa_0}(s_3-s_2)^{-\tau}(\|u\|_{s_3}+\varepsilon|T|_{s_1}\|u\|_{s_2}).
\end{eqnarray}
\end{lemma}
\begin{proof}
The Neumann series
\begin{eqnarray}\label{E4-16R}
\mathcal{L}^{-1}=(I+\mathcal{D}^{-1}\mathcal{T})^{-1}\mathcal{D}^{-1}
=\sum_{m\geq0}(-1)^m(\mathcal{D}^{-1}\mathcal{T})^m\mathcal{D}^{-1}
\end{eqnarray}
is totally convergent in operator norm $\|\cdot\|_{s_0}$ with $\|\mathcal{L}^{-1}\|_{s_0}\leq c\gamma_1^{-1}N^{\tau}$,
by using (\ref{E4-17}).

By (\ref{E4-19}) and (\ref{E4-16R}), we have
\begin{eqnarray}\label{E4-16}
\|\mathcal{L}^{-1}u\|_{s_1}&\leq&\|\mathcal{D}^{-1}u\|_{s_1}+
\sum_{m\geq1}\|(\mathcal{D}^{-1}\mathcal{T})^m\mathcal{D}^{-1}u\|_{s_1}\nonumber\\
&\leq&\|\mathcal{D}^{-1}u\|_{s_1}+
\|\mathcal{D}^{-1}u\|_{s_1}\sum_{m\geq1}(\varepsilon\gamma_1^{-1}K(s)|T|_{s_0})^m\nonumber\\
&&+N^{\kappa_0}K(s_1)\varepsilon\gamma_1^{-1}|T|_{s_1}\|\mathcal{D}^{-1}u\|_{s_0}\sum_{m\geq1}m(K(s)\varepsilon\gamma_1^{-1}|T|_{s_0})^{m-1}.~~~~~~~~~
\end{eqnarray}
Using $\sup_{x>0}(x^ye^{-x})=(ye^{-1})^y$, $\forall y\geq0$, for $0<s_1<s_2<s_3$, it follows from Lemma 20 that
\begin{eqnarray}\label{E4-20}
\|\mathcal{D}^{-1}u\|^2_{s_1}&=&\|\mathcal{D}^{-1}\bar{D}\bar{D}^{-1}u\|^2_{s_1}\leq c^2(\varsigma,s_1,\gamma_1)N^{2\tau}\|\bar{D}^{-1}u\|^2_{s_2}\nonumber\\
&=&c^2(\varsigma,s_1,\gamma_1)N^{2\tau}\sum_{j\in S\cap J_N^+}e^{2|j+\overrightarrow{\rho}|s_2}\|\bar{D}^{-1}u_j\|_{\textbf{L}^2}^2\nonumber\\
&\leq&c^2(\varsigma,s_1,\gamma_1)N^{2\tau}\sum_{j\in S\cap J_N^+}e^{2|j+\overrightarrow{\rho}|s_2}|\omega_j^2+1-\varepsilon a\omega_j^{2p}|^{-2}\|u_j\|_{\textbf{L}^2}^2\nonumber\\
&\leq&c^2(\varsigma,s_1,\gamma_1)N^{2\tau}\sum_{j\in S\cap J_N^+}e^{-2|j+\overrightarrow{\rho}|(s_3-s_2)}|j+\overrightarrow{\rho}|^{-2}e^{2|j+\overrightarrow{\rho}|s_3}\|u_j\|_{\textbf{L}^2}^2~~~~\nonumber\\
&\leq&c^2(\varsigma,\tau,s_1,\gamma_1,\gamma)N^{2\tau}(s_3-s_2)^{-2\tau}\|u\|_{s_3}^2.
\end{eqnarray}
Thus by (\ref{E4-16}) and (\ref{E4-20}), we derive
\begin{eqnarray}\label{E4-21}
\|\mathcal{L}^{-1}u\|_{s_1}
&\leq&\gamma_1^{-1}N^{\kappa_0}K'(s_1)(\|\mathcal{D}^{-1}u\|_{s_1}+\varepsilon|T|_{s_1}\|\mathcal{D}^{-1}u\|_{s_0})\nonumber\\
&\leq&c(\varsigma,\tau,s_1,\gamma_1,\gamma)N^{\tau+\kappa_0}(s_3-s_2)^{-\tau}(\|u\|_{s_3}+\varepsilon|T|_{s_1}\|u\|_{s_2}),~~~~~
\end{eqnarray}
where $0<s_1<s_2<s_3$ and $\varepsilon\gamma_1^{-1}\varsigma^{-1}(1+|T|_{s_0})\leq c(k)$ small enough.
\end{proof}
Now we are ready to prove Lemma 12. Let $u=u_R+u_S$ with $u_S\in\textbf{H}_S$, $u_R\in\textbf{H}_R$. Then by the resolvent identity (\ref{E4-22}),
\begin{eqnarray}\label{E4-23}
\|(L^{(N)})^{-1}u\|_{s_1}&\leq&\|L_R^{-1}u_R+L_R^{-1}L_S^R\mathcal{L}^{-1}(u_S+L_{R}^SL_R^{-1}u_R)\|_{s_1}
+\|\mathcal{L}^{-1}(u_R+L_R^SL_R^{-1}u_R)\|_{s_1}\nonumber\\
&\leq&\|L_R^{-1}u_R\|_{s_1}+\|L_R^{-1}L_S^R\mathcal{L}^{-1}u_S\|_{s_1}+\|L_R^{-1}L_S^R\mathcal{L}^{-1}L_{R}^SL_R^{-1}u_R\|_{s_1}\nonumber\\
&&+\|\mathcal{L}^{-1}u_R\|_{s_1}+\|\mathcal{L}^{-1}L_R^SL_R^{-1}u_R\|_{s_1}.
\end{eqnarray}
Next we estimate the right hand side of (\ref{E4-23}) one by one. Using (\ref{E2-6}), (\ref{E4-12}) and (\ref{E4-18}), for $0<s_1<s_2<s_3<s_4<k-1$, we have
\begin{eqnarray}\label{E4-24}
\|L_R^{-1}L_S^R\mathcal{L}^{-1}u_S\|_{s_1}&\leq&c(\gamma,\tau,s_2)(s_2-s_1)^{-\tau}(1+\varepsilon\varsigma^{-1}|T|_{s_2})\|L_S^R\mathcal{L}^{-1}u_S\|_{s_2}\nonumber\\
&\leq&c(\gamma,\tau,s_2)(s_2-s_1)^{-\tau}(1+\varepsilon\varsigma^{-1}|T|_{s_2})|T|_{s_2}\|\mathcal{L}^{-1}u\|_{s_2}\nonumber\\
&\leq&c(\gamma,\gamma_1,\varsigma,\tau,s_2)(s_2-s_1)^{-\tau}(s_4-s_3)^{-\tau}N^{\tau+\kappa_0}\nonumber\\
&&\times(1+\varepsilon\varsigma^{-1}|T|_{s_2})|T|_{s_2}(\|u\|_{s_3}+\varepsilon|T|_{s_2}\|u\|_{s_4}),
\end{eqnarray}
\begin{eqnarray}\label{E4-25}
\|\mathcal{L}^{-1}L_R^SL_R^{-1}u_R\|_{s_1}&\leq& c(\varsigma,\tau,s_1,\gamma_1,\gamma)N^{\tau+\kappa_0}(s_3-s_2)^{-\tau}\nonumber\\
&&\times(\|L_R^SL_R^{-1}u_R\|_{s_3}+\varepsilon|T|_{s_1}\|L_R^SL_R^{-1}u_R\|_{s_2})\nonumber\\
&\leq&c(\varsigma,\tau,s_1,s_2,s_3,\gamma_1,\gamma)N^{\tau+\kappa_0}(s_3-s_2)^{-\tau}\nonumber\\
&&\times(|T|_{s_3}\|L_R^{-1}u_R\|_{s_3}+\varepsilon|T|_{s_1}|T|_{s_2}\|L_R^{-1}u_R\|_{s_2})\nonumber\\
&\leq&c(\varsigma,\tau,s_1,s_2,s_3,\gamma_1,\gamma)N^{\tau+\kappa_0}(s_3-s_2)^{-\tau}\nonumber\\
&&\times(|T|_{s_3}(s_4-s_3)^{-\tau}(1+\varepsilon\varsigma^{-1}|T|_{s_4})\|u\|_{s_4}\nonumber\\
&&+\varepsilon|T|_{s_1}|T|_{s_2}(s_3-s_2)^{-\tau}(1+\varepsilon\varsigma^{-1}|T|_{s_3})\|u\|_{s_3})\nonumber\\
&\leq&c(\varsigma,\tau,s_1,s_2,s_3,\gamma_1,\gamma)N^{\tau+\kappa_0}(s_3-s_2)^{-\tau}|T|_{s_3}(1+\varepsilon\varsigma^{-1}|T|_{s_4})\nonumber\\
&&\times((s_4-s_3)^{-\tau}\|u\|_{s_4}+\varepsilon|T|_{s_2}(s_3-s_2)^{-\tau}\|u\|_{s_3}),
\end{eqnarray}
\begin{eqnarray}\label{E4-26}
\|L_R^{-1}L_S^R\mathcal{L}^{-1}L_{R}^SL_R^{-1}u_R\|_{s_1}&\leq& c(\gamma,\tau,s_2)(s_2-s_1)^{-\tau}(1+\varepsilon\varsigma^{-1}|T|_{s_2})\|L_S^R\mathcal{L}^{-1}L_{R}^SL_R^{-1}u_R\|_{s_2}\nonumber\\
&\leq&c(\gamma,\tau,s_2)(s_2-s_1)^{-\tau}(1+\varepsilon\varsigma^{-1}|T|_{s_2})|T|_{s_2}\|\mathcal{L}^{-1}L_{R}^SL_R^{-1}u_R\|_{s_2}\nonumber\\
&\leq&c(\varsigma,\tau,s_1,s_2,s_3,\gamma_1,\gamma)N^{\tau+\kappa_0}(s_3-s_2)^{-\tau}(s_2-s_1)^{-\tau}|T|^2_{s_3}\nonumber\\
&&\times(1+\varepsilon\varsigma^{-1}|T|_{s_4})^2((s_4-s_3)^{-\tau}\|u\|_{s_4}\nonumber\\
&&+\varepsilon|T|_{s_2}(s_3-s_2)^{-\tau}\|u\|_{s_3}).
\end{eqnarray}
The terms $\|L_R^{-1}u_R\|_{s_1}$ and $\|\mathcal{L}^{-1}u_R\|_{s_1}$ can be controlled by using (\ref{E4-12}) and (\ref{E4-18}). Thus by (\ref{E4-23})-(\ref{E4-26}), for $0<s<\tilde{s}$, we conclude
\begin{eqnarray*}
\|(L^{(N)})^{-1}u\|_{s}\leq c(\varsigma,\tau,s,\tilde{s},\gamma_1,\gamma)N^{\tau+\kappa_0}(1+\varepsilon\varsigma^{-1}|T|_{\tilde{s}})^3(\tilde{s}-s)^{-\tau}\|u\|_{\tilde{s}},
\end{eqnarray*}
which together with Lemma 18 gives (\ref{E3-4}).

\section{Appendix}
For completeness, we give the proof of Lemma 17 (Measure estimates) and Lemma 18, which follows essentially the scheme of \cite{Berti2,Berti1,Bourgain1}.

\textbf{Proof of Lemma 17}.
Note that $|j+\overrightarrow{\rho}|\leq r$ and the eigenvalue of the operator $L_a^{(r)}$ has the form $\omega^2_j+1-\varepsilon a\omega_j^{2\varrho}-O(\varepsilon)$ of the operator $L_a^{(r)}$. Here $j=(j_1,j_2)\in\Lambda^+\times\textbf{Z}^n$. For sufficient small $\varepsilon_0\gamma^{-1}M^{\tau+2\varrho}$, by (\ref{E1-5}), all the eigenvalues of $L_a^{(r)}$ has modulus $\geq\gamma(4r^{\tau})^{-1}\geq\gamma_1(4r^{\kappa})^{-1}$. Thus $\mathcal{G}_r=[0,\delta_0)$ and the measure estimate (\ref{E8-3}) for $\mathcal{G}$ is standard. To prove the measure estimate (\ref{E8-4}), we divide the process of proof into two cases. For the case $N,N'\leq N_{\varepsilon_0}:=(c\gamma_1\varepsilon_0^{-1})^{\frac{1}{\tau+2\varrho}}$, $\mathcal{G}_{\gamma_1,\kappa_0}^{(N')}(u_2)=\mathcal{G}_{\gamma_1,\kappa_0}^{(N)}(u_1)=\mathcal{G}$, by the same process of proof of (\ref{E8-3}), one can prove (\ref{E8-4}) holds. For other cases, it is sufficient to prove
\begin{eqnarray*}
|(\mathcal{G}_{\gamma_1,\kappa_0}^{(N')}(u_2))^c\backslash(\mathcal{G}_{\gamma_1,\kappa_0}^{(N)}(u_1))^c\cap[\frac{\delta_1}{2},\delta_1)|\leq C\gamma_1\delta N^{-1},~~\forall \delta_1\in[0,\delta_0].
\end{eqnarray*}
For fixed $\delta_1$ and the decomposition $[0,\delta_0]=\cup_{n\geq1}[\delta_02^{-n},\delta_02^{-(n-1)}]$, we consider the complementary sets in $[\frac{\delta_1}{2},\delta_1)$
\begin{eqnarray*}
(\mathcal{G}_{\gamma_1,\kappa_0}^{(N')}(u_2))^c\backslash(\mathcal{G}_{\gamma_1,\kappa_0}^{(N)}(u_1))^c&=&(\mathcal{G}_{\gamma_1,\kappa_0}^{(N')}(u_2))^c\cap\mathcal{G}_{\gamma_1,\kappa_0}^{(N)}(u_1)\nonumber\\
&\subset&[\cup_{r\leq N}(\mathcal{G}_r^c(u_2)\cap\mathcal{G}_r(u_1)\cap\mathcal{G})]\cup[\cup_{r>N}\mathcal{G}_r^c(u_2)\cap\mathcal{G}].
\end{eqnarray*}
If $r\leq N_{\varepsilon_0}$, then $\mathcal{G}_r^c(u_2)\cap\mathcal{G}=\emptyset$ . So it is sufficient to prove that, if $\|u_1-u_2\|_{\bar{\sigma}}\leq N^{-e}$, $e\geq d+n+3$, then
\begin{eqnarray*}
\Omega:=\sum_{N_{\varepsilon}<r\leq N}|\mathcal{G}_r^c(u_2)\cap\mathcal{G}_r(u_1)|+\sum_{r>\max\{N,N_{\varepsilon}\}}|\mathcal{G}_r^c(u_2)|\leq C'\gamma_1\delta_1 N^{-1}.
\end{eqnarray*}
Note that $\|(L^{(r)}_a)^{-1}\|_0$ is the inverse of the eigenvalue of smallest modulus and
\begin{eqnarray*}
\|L_a^{(r)}(u_2)-L_a^{(r)}(u_1)\|_0=O(\varepsilon\|u_2-u_1\|_{s_0})=O(\varepsilon N^{-e}).
\end{eqnarray*}
The sufficient and necessary condition of an eigenvalues of $L_a^{(r)}(u_2)$ in $[-4\gamma_1 r^{-\tau}-C\varepsilon N^{-e},4\gamma_1 r^{-\tau}+C\varepsilon N^{-e}]$ is that there exists an eigenvalues of $L_a^{(r)}(u_1)$ in  $[-4\gamma_1 r^{-\tau},4\gamma_1 r^{-\tau}]$. Thus, it leads to
\begin{eqnarray*}
\mathcal{G}_r^c(u_2)\cap\mathcal{G}_r(u_1)\subset\{\delta\in[\frac{\delta_1}{2},\delta_1]|&&\exists~at~leat~an~eigenvalue~of~L_a^{(r)}(\delta,u_1)\\
&&~with~modulus~in~[4\gamma_1 r^{-\tau},4\gamma_1 r^{-\tau}+C\varepsilon N^{-e}]\}.
\end{eqnarray*}
Next we claim that if $\varepsilon$ is small enough and $I$ is a compact interval in $[-\gamma_1,\gamma_1]$ of length $|I|$, then
\begin{eqnarray}\label{EAA1}
|\{\delta\in[\frac{\delta_1}{2},\delta]~s.t.~at~least&&\exists~an~eigenvalue~of~L^{(r)}(\delta,u_1)~belongs~to~I\}|\nonumber\\
&&\leq Cr^{d+n+1}\delta_1^{-(2\rho-2)}|I|.
\end{eqnarray}
Due to the $C^1$ map $\delta\mapsto L^{(r)}(\delta,u_1)$ and the selfadjoint property of $L^{(r)}(\delta,u_1)$, we have the corresponding eigenvalue function $\lambda_{k}(\delta,u_1)$ with $1\leq k\leq r$. Denote the eigenspace of $L^{(r)}(\delta,u_1)$ by $E_{\delta,k}$ associated to $\lambda_k(\delta,u_1)$, then by $\|\partial_{\delta}b\|_s=\|(\partial_{u}^2f)(x,u)\|_s\leq C\gamma_1^{-1}$ and $\|\nabla^{\rho}h\|_0^2\geq\|h\|_0^2$, for sufficient small $0<\varepsilon\leq\varepsilon_0(\gamma_1)$, we have
\begin{eqnarray*}
(\partial_{\delta}\lambda_k(\delta,u_1))&\leq&\max_{h\in E_{\delta,k},\|h\|_0=1}\left((\partial_{\delta}L^{(r)})(\delta,u_1)h,h\right)_0\\
&\leq&\max_{h\in E_{\delta,k},\|h\|_0=1}\left((2\rho-1)\delta^{2\rho-2}(\triangle^{\rho}h,h)_0+O(\varepsilon\gamma_1^{-1})\right)\\
&\leq&\max_{h\in E_{\delta,k},\|h\|_0=1}\left(-(2\rho-1)\delta^{2\rho-2}\|\nabla^{\rho}h\|_0^2+O(\varepsilon\gamma_1^{-1})\right)\\
&\leq&\max_{h\in E_{\delta,k},\|h\|_0=1}\left(-(2\rho-1)\delta^{2\rho-2}\|h\|_0^2+O(\varepsilon\gamma_1^{-1})\right)\\
&\leq&-(2\rho-1)\delta^{2\rho-2}+O(\varepsilon\gamma_1^{-1})\leq-2(\rho-1)\delta_1^{2\rho-2}.
\end{eqnarray*}
Hence we have $|\lambda_{k}^{-1}(I,u_1)\cap[\frac{\delta_1}{2},\delta_1]|\leq C|I|\delta_1^{-(2\rho-2)}$. The claim holds.

Thus, we obtain
\begin{eqnarray*}
|\mathcal{G}_r^c(u_2)\cap\mathcal{G}_r(u_1)|\leq C\varepsilon r^{d+n+1}\delta_1^{-(2\rho-2)}N^{-e}
\leq C\delta_1N^{-e}r^{d+n+1}.
\end{eqnarray*}
Furthermore, by (\ref{EAA1}), we have $|\mathcal{G}_r^c(u_2)|\leq C\gamma_1r^{d+n-\tau+1}\delta_1^{-(2\rho-2)}$. Therefore, we obtain
\begin{eqnarray*}
\Omega&=&\sum_{N_{\varepsilon}<r\leq N}|\mathcal{G}_r^c(u_2)\cap\mathcal{G}_r(u_1)|+\sum_{r>\max\{N,N_{\varepsilon}\}}|\mathcal{G}_r^c(u_2)|\\
&\leq&C\delta_1(\sum_{r\leq N}r^{d+n+1})N^{-e}+C\gamma_1\delta_1^{-(2\rho-2)}\sum_{r>\max\{N,N_{\varepsilon}\}}r^{d+n-\tau+1}\\
&\leq&C'\left(\delta_1N^{d+n-e+2}+\gamma_1\delta_1^{-(2\rho-2)}(\max\{N,N_{\varepsilon}\})^{d+n-\tau+2}\right)\\
&\leq&C''\gamma_1\delta_1N^{-1},
\end{eqnarray*}
where $C$, $C'$ and $C''$ denote constants. This completes the proof.

\textbf{Proof of Lemma 18}.
The key step of Lemma 18 is the following Theorem 8.  To prove Theorem 8, we only need to give the proof of Lemma 25, the remainder of the proof is the same as \cite{Berti1}, so we omit it.

Define the bilinear symmetric form $\phi_{\varepsilon}:\textbf{R}^{r+n}\times\textbf{R}^{r+n}\longrightarrow\textbf{R}$ by
\begin{eqnarray*}
\phi_{\varepsilon}(x,x'):=J\cdot J'-\varepsilon aJ^*\cdot J^{*'},~~\forall J\in\textbf{R}^{r+n},
\end{eqnarray*}
where $x=(J,J^*)$, $x'=(J',J^{*'})\in\textbf{R}^{r+n}\times\textbf{R}^{r+n}$ and choose $J^*\in\textbf{R}^{r+n}$ such that the corresponding quadratic form
\begin{eqnarray*}
\Phi_{\varepsilon}(x)=\phi_{\varepsilon}(x,x)=|J|^2-\varepsilon a|J|^{2\varrho}.
\end{eqnarray*}
Denote $x=j'+\overrightarrow{\rho'}=(J,J^*)$, where $\overrightarrow{\rho'}=(\rho,0,0,0)$, $\forall j'=(j_1,j_2,j^*_1,j_2^*)\in\Lambda^+\times\textbf{Z}^n\times\Lambda^+\times\textbf{Z}^n$ and $x\in\Lambda^{++}\times\textbf{Z}^n\times\Lambda^+\times\textbf{Z}^n$ since $j_1\in\Lambda^+$ and $\Lambda^{++}=\rho+\Lambda^+$.
Note that $\Phi_{\varepsilon}(j'+\overrightarrow{\rho'})=D_j+|\rho|^2$, where $D_j$ are the small divisors. We say a vector $x=(j'+\rho,j^{*'})\in\Lambda^{++}\times\textbf{Z}^n\times\Lambda^{++}\times\textbf{Z}^n$ is ''weak singular'' if $|\Phi_{\varepsilon}(x)|\leq C$ for some constant $C$ fixed once and for all.

\textbf{Definition 3.}
A sequence $x_0,x_1,\cdots,x_K\in\Lambda^{++}\times\textbf{Z}^n\times\Lambda^{++}\times\textbf{Z}^n$ of distinct, weakly singular vectors satisfying, for some $B\geq2$, $|x_{k+1}-x_k|\leq B$, $\forall k=0,1,\cdots,K-1$, is called a $B$-chain of length $K$.

\begin{theorem}
Assume that $\varepsilon$ satisfies (\ref{E1-4}). Then any $B$-chain has length $K\leq B^C\gamma^{-p}$ for some $C:=C(G)>0$ and $p:=p(G)>0$.
\end{theorem}
Using Lemma 2, we can easily prove the following result. It can be found in \cite{Berti1}.
\begin{lemma}
Let $\mathcal{M}=(G\times\textbf{T}^n)/N$. The matrices $R$ and $S$ have coefficient in $D^{-1}\textbf{Z}$ for some $D\in\textbf{N}$.
\end{lemma}
Given lattice vectors $f_i\in\Lambda\times\textbf{Z}^n\times\Lambda\times\textbf{Z}^n$, $i=1,\cdots,n$, $1\leq m\leq 2r+2n$, linearly dependent on $\textbf{R}$, we consider the subspace $F:=Span_{\textbf{R}}\{f_1,\cdots,f_m\}$ of $\textbf{R}^{r+n}\times\textbf{R}^{r+n}$ and the restriction $\phi_{\varepsilon}|_{F}$ of the bilinear form $\phi_{\varepsilon}$ to $F$, which is represented by the symmetric matrix $A_{\varepsilon}:=\{\phi_{\varepsilon}(f_i,f_{i'})\}_{i,i'=1}^m$. Denote $\varphi(x,x'):=J\cdot J'$ and $\varphi^*(x,x'):=J^*\cdot J^{*'}$ the symmetric bilinear forms. Then we rewrite
\begin{eqnarray*}
\Phi_{\varepsilon}=\varphi-\varepsilon a\varphi^*,~~A_{\varepsilon}=R-\varepsilon aS,
\end{eqnarray*}
where
\begin{eqnarray*}
R:=\{\varphi(f_i,f_{i'})\}_{i,i'=1}^m=(a_{ii'})_{i,i'=1}^m,~~S:=\{\varphi^*(f_i,f_{i'})\}_{i,i'=1}^m=(b_{ii'})_{i,i'=1}^m
\end{eqnarray*}
 are the matrices that represent, respectively, $a_{ii'}$ and $b_{ii'}$, $i,i'=1,\cdots,m$ denote the element, respectively, of $R$ and $S$,
$\varphi|_{F}$ and $\varphi^*|_{F}$ in the basis $\{f_1,\cdots,f_m\}$.

Since the matrix $S$ is not at most rank 1, the proof of following result is some what different from the proof of Lemma A.3 in \cite{Berti1}. But the main idea is the same.
\begin{lemma}
Assume that $a$ satisfies (\ref{E1-4}). Then $A_{\varepsilon}$ satisfies
\begin{eqnarray*}
\|A_{\varepsilon}^{-1}\|\leq\frac{c(m,D)}{\gamma}(\max_{i=1,\cdots,m}|f_i|)^{5m-2},
\end{eqnarray*}
where $c(m,D)$ is a constant depending on $m$ and $D$.
\end{lemma}
\begin{proof}
Direct calculation shows
\begin{eqnarray}\label{E5-1}
&&\det A_{\varepsilon}=\det(R-\varepsilon aS)\nonumber\\
&=&\sum_{j_1,j_2,\cdot,j_m}(-1)^{\tau(j_1,j_2,\cdot,j_m)}(a_{1j_1}-\varepsilon ab_{1j_1})(a_{2j_2}-\varepsilon ab_{2j_2})\cdots (a_{mj_m}-\varepsilon ab_{mj_m})\nonumber\\
&=&\det R+(-1)^n\varepsilon^ma^m\det S+P'(\varepsilon),
\end{eqnarray}
where $P'(\varepsilon)$ is a polynomial on $\varepsilon$ of degree $m-1$ with integer coefficients (by Lemma 23), and $\tau(j_1,j_2,\cdot,j_m)$ is the rank of $j_1,\cdots,j_m$.

Note that $\det R,\det S,P'(1)\in D^{-m}\textbf{Z}$. By Lemma 23,
$D^m\det A_{\varepsilon}=P(\varepsilon)$ is a polynomial on $\varepsilon$ of degree $n$ with integer coefficients. It follows from $P(-a)=D^m\det(R+a^2S)$ that $P(\cdot)\neq0$.
By (\ref{E5-1}), if $\det R=0$, then $|\det A_{\varepsilon}|\geq\varepsilon^na^nD^{-m}$. If $\det R\neq0$, then by (\ref{E5-1R}), we have
\begin{eqnarray}\label{E5-2}
|\det A_{\varepsilon}|\geq\gamma D^{\frac{5m}{2}}|\det R|^{-\frac{3}{2}}.
\end{eqnarray}
We can write $R+a^2S=\xi^T\xi$ with $\xi=(f_1,\cdots,f_m)$. Thus we have
\begin{eqnarray}\label{E5-3}
0\leq\det R\leq\det(R+S)=(\det\xi)^2\leq|f_1|^2\cdots|f_m|^2\leq M^{2m},
\end{eqnarray}
where $M:=\max_{i=1,\cdots,m}|f_i|$.

By(\ref{E5-2})-(\ref{E5-3}), we derive
\begin{eqnarray}\label{E5-4}
|\det A_{\varepsilon}|\geq\gamma D^{\frac{5m}{2}}|\det R|^{-3m}.
\end{eqnarray}
Note that (\ref{E5-1R}), $a\geq\gamma$ and (\ref{E5-4}) hold. Using the Cramer rule and (\ref{E5-4}), we can obtain the main result.
This completes the proof.
\end{proof}

\begin{acknowledgements}
This work was done in Beijing International Center for Mathematical Research, Peking University.
The first author expresses his sincere thanks to Prof Gang Tian for his suggestion and encouragement!
We also express our sincerely thanks to prof P.H. Rabinowtiz for his many kind suggestions and help, and inform us the papers of T. Kato\cite{K} and Q. Han, J.X. Hong and C.S. Lin\cite{H}. This first author is supported by NSFC Grant 11201172, Post-doctor fund 2012M510243 and 985 project of Jilin
University. The second author is supported by NSFC Grant 10531050, National 973 Project of
China 2006CD805903, SRFDP Grant 20040183030, the 985 Project of Jilin
University.
\end{acknowledgements}

\end{document}